\documentclass[12pt]{article}
\usepackage{geometry} 
\usepackage{amsmath,amsthm,amssymb}                
\usepackage{graphicx, color}
\usepackage{amssymb}
\usepackage{epstopdf}
\usepackage{pdfsync}
\usepackage{mathabx}

\newtheorem{definition}{Definition}[section]
\newtheorem{lemma}[definition]{Lemma}
\newtheorem{theorem}[definition]{Theorem}
\newtheorem{proposition}[definition]{Proposition}

\newtheorem{remark}[definition]{Remark}

\RequirePackage{authblk}

\newcommand{\R}{\mathbb{R}}
\newcommand{\N}{\mathbb{N}}

\newcommand{\BV}{\mathrm{BV}}

\numberwithin{equation}{section}

\def\e{\varepsilon}

\def\R{\mathbb R}

\def\N{\mathbb N}

\begin{document}

\author{Andrea Braides\footnote{Department of Mathematics, University of Rome Tor Vergata, via della Ricerca Scientifica 1, 00133 Rome, Italy} \ and Abdilaziz Zoir Ugli Komilov\footnote{SISSA, via Bonomea 265, 34146 Trieste, Italy}}
\title{Homogenization in fractional phase transitions 
\\ at the critical scale}
\date{}
\def\epsilon{\varepsilon}
\maketitle

\begin{abstract}
We analyze fractional phase-transition energies with periodic hetero\-genei\-ties at the critical $H^{1/2}$
 scaling. We prove that the $\Gamma$-limit is a sharp-interface functional whose surface energy density combines homogenization and averaging effects. The resulting coefficient is a weighted combination of the minimum and the mean of the oscillatory parameter, reflecting the coexistence of multiple interaction scales. This behavior is specific to the critical regime and differs from the case $s>1/2$.

 {\bf Keywords:} $\Gamma$-convergence, homogenization, fractional phase transitions, nonlocal energies, sharp-interface limits, critical scaling.
 
 {\bf MSC 2020:} 49J45, 35B27, 35R11, 49M25, 26A33.
 
 \end{abstract}
\section{Introduction}
Fractional phase-transition energies display qualitatively different asymptotic regimes, depending on the order of the fractional seminorm. In the critical case $s=1/2$, the energy scaling becomes logarithmic and leads to sharp-interface limits independent of the potential.
The aim of this paper is to understand how the critical regime interacts with oscillatory heterogeneities and whether homogenization effects survive at this scale or are replaced by different effective behaviors.
We show the coexistence of homogenized and averaged contributions in the sharp-interface limit, which is a novel phenomenon for heterogeneous energies.

\smallskip
The study of fractional phase transitions at the critical scaling in the framework of $\Gamma$-convergence dates back to the work of Alberti et al.~\cite{ABS} (see also \cite{AB}), who examined the asymptotic behaviour of energies
\begin{equation}
    \frac{1}{\varepsilon|\log{\varepsilon}|}\int_{0}^{1}W(u(t))dt+\frac{1}{|\log\varepsilon|}\int_{0}^{1}\int_{0}^{1}\frac{\left|u(s)-u(t)\right|^2}{\left|s-t\right|^2}ds\,dt.
\end{equation}
The first term contains a double-well potential $W$, which we can assume to vanish at $-1$ and $1$, while the second term is the square of the $H^{1/2}$-seminorm. Alberti et al.~showed that the $\Gamma$-limit of such energies is
a sharp-transition energy \begin{equation}\label{shat} 
8 \#(S(u)),
\end{equation} 
with domain $BV((0,1);\{-1,1\})$; that is, defined on piecewise-constant functions taking the values in the wells, where $S(u)$ denotes the set of jump points of $u$. Note that the value of the limit is independent of $W$. This behavior should be contrasted with the case of $H^s$ fractional phase transitions with $s>1/2$. In that regime, the correct scaling is $1/\varepsilon$ for all $s$,
and the sharp-transition limit depends on the potential $W$ through an optimal-profile formula (see \cite{SV,Solci2,PaVi}). In this sense, the behaviour at $s=1/2$ is critical, showing a transition from a $1/\varepsilon$ scaling to a $1/\varepsilon|\log\varepsilon|$
scaling. This transition has been recently analyzed by Picerni \cite{Pic} in the framework of the theory of $\Gamma$-expansions \cite{BT}, studying all possible $s$-$\varepsilon$ regimes when $s\to1/2$. For a general discussion on nonlocal phase transitions, see also \cite{DipVal}.

The singular behaviour is characteristic of scale-invariant energies such as the $H^{1/2}$-seminorm. Similar behaviours can be noticed for capacitary potentials or for vortex energies \cite{BBH,ABO}.
In those frameworks, an interesting issue is the interaction with oscillations in the perturbation term. In our context, the prototypical energies to study are of the form 
\begin{equation}
F_\varepsilon(u)=    \frac{1}{\varepsilon|\log{\varepsilon}|}\int_{0}^{1}W(u(t))dt+\frac{1}{|\log\varepsilon|}\int_{0}^{1}\int_{0}^{1}a\Big(\frac{s}\delta,\frac{t}\delta
    \Big)\frac{\left|u(s)-u(t)\right|^2}{\left|s-t\right|^2}ds\,dt,
\end{equation}
with $a$ a continuous and strictly positive coefficient, $1$-periodic in both variables. This work provides a first extension of heterogeneous phase-transition results to the critical fractional regime, (see \cite{Ansini_Braides_Chiado_2003,MR4633768,MR4898688,MR4014392} for analogous local problems).

Our main result describes the sharp-interface limit of $F_\e$.
To that end, we assume that $\delta=\delta_\e$ is such that $\delta_\e\to 0$ as $\e\to 0$, and that there exist
$$
\lambda_0=\lim_{\e\to 0}\frac{\log\delta}{\log\epsilon}
$$
(which we may assume exists up to subsequences). We then set $\lambda=\min\{\lambda_0,1\}$.
Under these assumptions, the $\Gamma$-limit of $F_\e$
is given by 
\begin{equation}\label{sht}
    8\big((1-\lambda)a_{\rm min}+\lambda \overline a\big) \#(S(u)),
\end{equation}
where 
\begin{equation}
    a_{\rm min}=\min_t a(t,t),\qquad
    \overline a=\int_0^1\int_0^1 a(t,s)\,dtds.
\end{equation}
Formula \eqref{sht} reflects the coexistence of two different interaction regimes contributing to each jump point. Interactions occurring at scales much smaller than $\delta$ effectively sample local values of $a$, and therefore select the minimal value $a_{\rm min}$ of $a$ on the diagonal. Interactions occurring at scales much larger than $\delta$, instead, average out the oscillations and contribute through the mean value $\overline{a}$. The weights $1-\lambda$ and $\lambda$ correspond to the relative contributions of these two scale ranges. This coexistence of homogenized and averaged contributions is specific to the critical logarithmic scaling.

We note that \eqref{sht} can be interpreted as the homogenized energy of the one-dimensional free-discontinuity functional
\begin{equation}\label{fsh}
F^\delta(u)=
8\sum_{t\in S(u)} 
\Big((1-\lambda) a
\Big(\frac{t}\delta,\frac{t}\delta\Big)+\lambda \overline a\Big)
\qquad u\in BV((0,1);\{-1,1\}),
\end{equation}
so that the form of functional \eqref{sht} highlights the coexistence of homogenization and averaging effects.
This coexistence of different types of interaction for optimal sequences also appears in other variational problems, but with some differences highlighted in the resulting effective coefficients. For vortex problems in a heterogeneous medium, for example, the oscillation leads to homogenized coefficients, and not to averaged ones \cite{MR4367908}, while for capacitary problems, the final coefficient is obtained as the harmonic mean, and not the arithmetic mean, of the minimum value of the coefficient and of the (determinant of the) homogenized matrix \cite{MR4668552}.

Our result can be compared with the study of the interaction of homogenization and fractional phase transitions in the supercritical case \cite{CarScaVog}. In that case, the relevant scale of the interactions is of order $\varepsilon$, giving a critical behaviour when $\varepsilon$ is also the scale of the oscillations. In the two extreme regimes of the supercritical case, the limiting sharp-interface energy depends exclusively on $a_{\rm min}$ in one regime and exclusively on $\overline a$ in the other.

In higher dimension, one formally expects the term involving $a_{\rm min}$
 in \eqref{sht} to be replaced by an anisotropic perimeter arising from the homogenization of a free-discontinuity functional analogous to \eqref{fsh}; see \cite{AmBra,BMS}. Such a characterization would likely require a heterogeneous version of the fractional phase-transition result in \cite{SV}, followed by letting $\e\to 0$ with $\delta$ fixed.

\section{Statement of the main result}
Let $I$ be an open interval. We recall that a function $u$ belongs to the {\em fractional Sobolev space} $H^{1/2}(I)$ if $u \in L^{2}(I)$ and
\[
[u]_{H^{1/2}(I)}^{2} := \int_{I}\int_{I}
\frac{|u(s)-u(t)|^2}{|s-t|^2} ds\,dt < \infty.
\]
The quantity $[u]_{H^{1/2}(I)}$ is called the {\em$H^{1/2}$-Gagliardo seminorm} of $u$. We refer to \cite{leofrac} and \cite{Hitchhiker} for an introduction to fractional Sobolev spaces $H^s$.

\smallskip
We will study a sequence of non-convex integral functionals with $H^{1/2}$-per\-tur\-ba\-tions. 
In the following,  $W:\R\to[0,+\infty)$ will be a continuous function that satisfies the following assumptions:
\begin{itemize}
    \item [(A1)] ($2$-growth condition) there exist $c_1,c_2>0$ such that
    $W(x)\geq c_1\left|x\right|^2-c_2$ for all $x\in\mathbb R$.
    \item [(A2)] $W(x)=0$ if and only if $x\in\{-1,1\}$. 
\end{itemize}
This type of function is sometimes called a {\em double-well potential} with wells in $\{-1,1\}$.

Oscillations will be modeled by a periodic coefficient $a$ that satisfies the following condition.

\begin{itemize}
    \item [(A3)] The function $a:\R^{2}\to[\alpha_{a},\beta_{a}]$ is continuous, symmetric, and $1$-periodic in both variables, with $\alpha_{a}>0$.
    \end{itemize}
    
    
    Note that the assumption that $a$ is symmetric is not restrictive, up to replacing $a$ with $a_{\rm sym}$, defined as
    $
    a_{\rm sym}(s,t):=\frac{1}{2}\big(a(s,t)+a(t,s)\big)
    $.
   Indeed, after this substitution, the values of the integrals remain unchanged.

Given $\delta=\delta(\varepsilon)$, we will characterize the $\Gamma$-limit of the functionals 
\begin{equation}\label{Functional}
F_\varepsilon(u)=    \frac{1}{\varepsilon|\log{\varepsilon}|}\int_{0}^{1}W(u(t))dt+\frac{1}{|\log\varepsilon|}\int_{0}^{1}\int_{0}^{1}a\Big(\frac{s}{\delta},\frac{t}{\delta}
    \Big)\frac{\left|u(s)-u(t)\right|^2}{\left|s-t\right|^2}dsdt,
\end{equation}
with respect to the $L^2$-convergence. Here, $\varepsilon$ is the transition scale parameter, while 
$\delta$ is the oscillation period, both tending to $0$. It is not restrictive, up to subsequences, to suppose that the limit
\[
\lim_{\varepsilon\to0}\frac{\log{\delta(\varepsilon)}}{\log{\varepsilon}}=\lambda_{0}
\]
exists. Under this assumption, the main result of the paper is the following theorem.

\begin{theorem}\label{main}
Let $W:\R\to[0,+\infty)$ be a double-well potential that satisfies assumptions {\rm (A1)} and {\rm (A2)}, and let $a:\R^2\to[\alpha_{a},\beta_{a}]$ satisfy assumption {\rm (A3)}. Then the following holds.

{\rm(i) (Compactness)} 
    every sequence such that $\sup F_{\varepsilon}(u_{\varepsilon})<\infty$ is strongly relatively compact in $L^2(0,1)$ and every limit point belongs to $\BV\left((0,1);\{-1,1\}\right)$;
    
{\rm(ii) ($\Gamma$-limit)} let $\lambda=\min\{\lambda_{0},1\}$, and let
$$
a_{\min}=\min\{a(t,t): t\in [0,1]\}, \qquad \bar{a}=\int_0^1\int_0^1 a(t,s)dtds;
$$
then the functional 
\begin{equation}\label{Gamma limit}
    F(u)= \begin{cases}
    8\left((1-\lambda)
    a_{\min}+\lambda \bar{a}\right)\#S(u) & \hbox{if } u\in BV((0,1);\{-1,1\})\\ 
    +\infty& \hbox{otherwise}
    \end{cases}
\end{equation}
is the $\Gamma$-limit of $F_\varepsilon$ 
with respect to the 
$L^2$-convergence.
\end{theorem} 

\begin{remark}[heterogeneities in the potential part]\rm
It is worth noting the asymmetry of the effect of heterogeneities when considered in the potential part as in \cite{MR4633768,MR4898688,MR4014392}. Indeed, if we take energies
$$
F_\varepsilon(u)=    \frac{1}{\varepsilon|\log{\varepsilon}|}\int_{0}^{1}a\Big(\frac t\delta\Big)W(u)dt+\frac{1}{|\log\varepsilon|}\int_{0}^{1}\int_{0}^{1}\frac{\left|u(s)-u(t)\right|^2}{\left|s-t\right|^2}ds\,dt,
$$
with $a$ periodic and any $\delta=\delta(\varepsilon)$, then the limit is always equal to the functional in \eqref{shat}, by the observed independence on the potential part. 
\end{remark}


\section{Compactness}
The compactness result in Theorem \ref{main}(i) has already been proved in \cite{ABS}.
In this section, we present a proof that will also be used in the computation of the $\Gamma$-limit.
The argument is based on a discretization procedure and on the following lemma, whose proof can also be found in \cite{ABS}. This lemma is used to estimate the contribution of non-local phase transitions on small intervals.

\begin{lemma}\label{lower bound lemma}
Let $\tau\in(0,\frac{1}{8})$, $u_\varepsilon \in L^2(0,1)$, and $I$ be an interval contained in $(0,1)$.  
For $\varepsilon>0$, $A(\varepsilon)$ and $B(\varepsilon)$ the sets defined as
\begin{equation*}
     A({\varepsilon}):=\big\{t\in (0,1):\,u_{\varepsilon}(t)\le\tau-1\big\},\,\,\, B({\varepsilon}):=\big\{t\in (0,1):\,u_{\varepsilon}(t)\ge1-\tau\big\},
 \end{equation*}
and set
$\displaystyle a(\varepsilon) := \frac{|A({\varepsilon}) \cap I|}{|I|}$,
and
$\displaystyle b(\varepsilon) := \frac{|B({\varepsilon}) \cap I|}{|I|}$.
Then, we have:
\begin{eqnarray}
\label{lover bound ineq}\nonumber
&&\hskip-1.6cm\bar{F}_{\varepsilon}(u,I):=
\frac{1}{\varepsilon\left|\log\varepsilon\right|}\int_{I} W(u)\,dt+\frac{1}{\left|\log\varepsilon\right|}
\int_{I}\int_{I}
\frac{\left|u(s)-u(t)\right|^2}{\left|s-t\right|^2}
dsdt\\
&&\ge
\frac{8(1-\tau)^2}{\left|\log\varepsilon\right|}
\Big(
\log(a(\varepsilon) b(\varepsilon))
+|\log\varepsilon|+\log{|I|}+\log\Big(\frac{\sigma(\tau)}{8(1-\tau)^2}\Big)
\Big),
\end{eqnarray}
where $\sigma(\tau):=\min\{W(x): ||x|-1|\ge\tau\}$. Note that
$\sigma(\tau)>0$ by Hypothesis {\rm (A2)}.
\end{lemma}

\subsection{A discretization procedure}\label{discr}
In this section, we introduce the relevant notation for a discretization argument that is used in the proof of the compactness result and will also be used in the computation of the $\Gamma$-limit.

Let $\theta\in[0,1)$, and let $\{\theta(\varepsilon)\}_{\varepsilon}>0$ be a sequence that converges to $0$ as $\varepsilon\to0$ such that 
\begin{equation}
\label{theta-eps}
\lim_{\varepsilon\to0}\frac{\log{\theta(\varepsilon)}}{\log{\varepsilon}}=\theta.
\end{equation}

The trivial choice for $\theta(\varepsilon)$ is $\varepsilon^\theta$, but we keep a more general notation since, in a periodic setting, we could need $\theta(\varepsilon)$ to be a multiple of $\delta$.

 Then, we consider a partition of $[0,1]$ into sets
\[
I_{k}({\varepsilon})=\begin{cases}
    [(k-1)\theta(\varepsilon),k\theta(\varepsilon)) & k\in\{1,...,n(\varepsilon)\}\\
    [n(\varepsilon)\theta(\varepsilon),1) & k=n(\varepsilon)+1,
\end{cases}
\]
where $n(\varepsilon)=\lfloor\frac{1}{\theta(\varepsilon)}\rfloor$. 

For $\tau\in(0,\frac{1}{8})$ and $k\in\{1,...,n(\varepsilon)\}$ we define the following sets:
\begin{equation*}
A_k(\varepsilon)=\big\{t\in I_{k}({\varepsilon}):\,u_{\varepsilon}(t)\le\tau-1\big\},\,\,\,B_k(\varepsilon)=\big\{t\in I_{k}({\varepsilon}):\,u_{\varepsilon}(t)\ge1-\tau\big\},
\end{equation*}
 and let $a_k(\varepsilon)$ and  $b_k(\varepsilon)$ be the Lebesgue measure of $A_k(\varepsilon)$ and $B_k(\varepsilon)$, respectively. Let $P_{1}(\varepsilon)$, $P_{2}(\varepsilon)$, $P_{3}(\varepsilon)$ and $\bar{P}_{3}(\varepsilon)$ be subsets of the set of indices $\{1,...,n(\varepsilon)\}$ defined by
    \[ 
     P_1(\varepsilon)=\left\{k\in\{1,...,n(\varepsilon)\}:a_k(\varepsilon)\ge(1-2\tau)\theta(\varepsilon)\right\},
    \]
    \[
     P_{2}(\varepsilon)=\left\{k\in\{1,...,n(\varepsilon)\}:b_k(\varepsilon)\ge(1-2\tau)\theta(\varepsilon)\right\},
    \]
    \[
     P_{3}(\varepsilon)=\left\{k\notin P_{1}(\varepsilon)\cup P_{2}(\varepsilon):\min{(a_k(\varepsilon),b_k(\varepsilon))}\ge\tau\theta(\varepsilon)\right\},
    \]
and let
\[
\bar{P}_{3}(\varepsilon)=\big\{k\in\{1,...,n(\varepsilon)\}:k\in P_{1}(\varepsilon),(k+1)\in P_{2}(\varepsilon)\,\,\text{or}\,\,k\in P_{2}(\varepsilon),k+1\in P_{1}(\varepsilon)\big\}. 
\]

We claim that $\#\{\bar{P}_{3}(\varepsilon)\cup P_{3}(\varepsilon)\}$ is uniformly bounded independently of $\varepsilon$ and $\tau$.
\begin{proposition}\label{3.2}
Let $(u_{\varepsilon})_{\varepsilon}\in H^{1/2}(0,1)$ be a sequence, and let $C>0$ be such that $\sup_{\varepsilon>0} F_{\varepsilon}(u_{\varepsilon}) \le C < \infty.$ Then the following holds.

\smallskip
{\rm(a)} For any $\tau\in(0,\frac{1}{8})$, there is $\varepsilon(\tau)>0$ such that 
\[
P_{1}(\varepsilon)\cup P_{2}(\varepsilon)\cup P_{3}(\varepsilon)=\{1,...,n(\varepsilon)\}
\]
for all $\varepsilon\in(0,\varepsilon(\tau))$;

\smallskip
{\rm(b)} $\#\{\bar{P}_{3}(\varepsilon)\cup P_{3}(\varepsilon)\}$ is uniformly bounded, independent of $\varepsilon$ and $\tau$.  
\end{proposition}
\begin{proof} 
(a) Suppose $k\notin P_{1}(\varepsilon)\cup P_{2}(\varepsilon)\cup P_{3}(\varepsilon)$. Then we have 
$\min(a_{k}(\varepsilon),b_{k}(\varepsilon))<\tau\theta(\varepsilon)$. 
We may assume that 
$\min(a_{k}(\varepsilon),b_{k}(\varepsilon))=a_{k}(\varepsilon)<\tau\theta(\varepsilon)
$,
and we also know that $k\notin P_{2}(\varepsilon)$. This implies that 
\[
|(A_{k}(\varepsilon)\cup B_{k}(\varepsilon))^{c}|=\theta(\varepsilon)-a_{k}(\varepsilon)-b_{k}(\varepsilon)\ge\tau\theta(\varepsilon),
\]
where $(A_{k}(\varepsilon)\cup B_{k}(\varepsilon))^{c}=\{t\in I_{k}(\varepsilon):\,\tau-1<u_{\varepsilon}(t)<1-\tau\}$. On the other hand, we have $
\inf_{x\in(\tau-1,1-\tau)}W(x)\ge\sigma(\tau)$, so that we obtain 
\[
\frac{\sigma(\tau)\tau\theta(\varepsilon)}{\varepsilon|\log{\varepsilon}|}\le\frac{1}{\varepsilon|\log{\varepsilon}|}\int_{(A_{k}(\varepsilon)\cup B_{k}(\varepsilon))^{c}}W(u_{\varepsilon})dt\le C.
\]
Since $\theta<1$, the left-hand side tends to $+\infty$ as $\varepsilon\to0$. 

\smallskip
(b) To prove claim (b), we use Lemma \ref{lower bound lemma}.
Indeed, let  $k\in \bar{P}_{3}(\varepsilon)\cup P_{3}(\varepsilon)$. By the lemma, if $k\in P_{3}(\varepsilon)$ then we obtain the lower bound 
\begin{equation*}
    F_{\varepsilon}(u_\varepsilon,I_{k}(\varepsilon))
\ge
\frac{8(1-\tau)^2\alpha_{a}}{\left|\log\varepsilon\right|}
\Big(
2\log\tau
+
|\log\varepsilon|
+\log{\theta(\varepsilon)}+
\log
\Big(
\frac{\sigma(\tau)}
{8(1-\tau)^2\beta_{a}}
\Big)
\Big),
\end{equation*}
while, if $k\in\bar{P}_{3}(\varepsilon)$ then
\begin{multline*}
     F_{\varepsilon}(u_\varepsilon,I_{k}(\varepsilon)\cup I_{k+1}(\varepsilon))
\ge
\frac{8(1-\tau)^2\alpha_{a}}{\left|\log\varepsilon\right|}
\Big(2\log{\Big(\frac{1}{2}-\tau\Big)}
+|\log\varepsilon|
\\+\log{2\theta(\varepsilon)}+
\log\Big(\frac{\sigma(\tau)}{8(1-\tau)^2\beta_{a}}\Big)\Big).
\end{multline*}

As a result, in both cases, we obtain the following lower bound
\begin{equation*}
     \liminf_{\varepsilon\to0}F_{\varepsilon}(u_\varepsilon,\bar{I}_{k}(\varepsilon))\ge8(1-\tau)^2(1-\theta)\alpha_{a}\ge(1-\theta)\alpha_{a}>0,
\end{equation*}
where 
\[
\bar{I}_{k}(\varepsilon)=\begin{cases}
    I_{k}(\varepsilon) & \hbox{if}\,\, k\in P_{3}(\varepsilon),\\
    I_{k}(\varepsilon)\cup I_{k+1}(\varepsilon) & \hbox{if}\,\, k\in \bar{P}_{3}(\varepsilon).
\end{cases}
\]
Since $\theta<1$ and $\alpha_{a}>0$, the uniform boundedness of $\left(F_{\varepsilon}(u_\varepsilon)\right)_{\varepsilon}$ implies that $\#\{P_{3}(\varepsilon)\cup \bar{P}_{3}(\varepsilon)\}$ is uniformly bounded independent of $\varepsilon$ and $\tau$.  
\end{proof}

\subsection{Proof of the compactness result}
Before proving the compactness result, we state and prove a simple estimate showing the convergence in measure.

\begin{proposition}\label{convergence in measure}
   Let $(u_{\varepsilon})_{\varepsilon}$ be a sequence and let $C>0$ be such that $\sup_{\varepsilon>0}F_{\varepsilon}(u_{\varepsilon})\le C<\infty$. Then, for every $\tau>0$ we have
\[
\lim_{\varepsilon\to0}|\{t\in(0,1):\,\operatorname{dist}(u_{\varepsilon}(t),\{-1,1\})\ge\tau\}|=0
\]
\end{proposition}

\begin{proof}
Let $\tau>0$ and let $\sigma(\tau)>0$ be defined as in Proposition 
\ref{lower bound lemma}.
Consequently,
\[
|\{t\in(0,1):\,\operatorname{dist}(u_{\varepsilon}(t),\{-1,1\})\ge\tau\}|\le\frac{1}{\sigma(\tau)}\int_{0}^{1}W(u_{\varepsilon})dt\le\frac{C\varepsilon|\log{\varepsilon}|}{\sigma(\tau)}.
\]
If we pass to the limit on both sides as $\varepsilon\to0$, we obtain the claim.
\end{proof}
We are now ready to state and prove the compactness result.

\begin{proposition}\label{compactness}
Let $(u_{\varepsilon})_{\varepsilon}$ be a sequence and let $C>0$ be such that $\sup_{\varepsilon>0}F_{\varepsilon}(u_{\varepsilon})\le C<\infty$. Then, there exists a subsequence $(u_{\varepsilon_{i}})$ such that $u_{\varepsilon_{i}}$ converges in $L^{2}(0,1)$ to some $u\in\BV((0,1);\{-1,1\})$.
\end{proposition}

\begin{proof}
Let $\{I_{k}(\varepsilon)\}_{k=1}^{n(\varepsilon)+1}$ be the partition defined above. With respect to this partition, we define the sequence of functions $(v^{\varepsilon,\tau})$ as follows:
\begin{equation*}
v^{\varepsilon,\tau}(t) =
\begin{cases}
-1 & \hbox{if}\,\, t \in I_{k}({\varepsilon})\,\, k \in P_{1}(\varepsilon), \\
1 & \hbox{if}\,\, t \in I_{k}({\varepsilon})\,\, k \in P_{2}(\varepsilon)\cup P_{3}(\varepsilon)\cup\{n(\varepsilon)+1\}.
\end{cases}
\end{equation*}
We now claim that, for every fixed $\tau \in (0,1/8)$, the sequence $(v^{\varepsilon,\tau})$ admits a subsequence that converges almost everywhere in $(0,1)$ to a function $v^{\tau}\in BV((0,1);\{-1,1\})$.

The uniform bound on $\#(P_{3}(\varepsilon)\cup \bar{P}_{3}(\varepsilon))$ implies that the number of jump points of $v^{\varepsilon,\tau}$ is uniformly bounded with respect to $\varepsilon$ and $\tau$. Hence, passing to a subsequence if necessary, we may assume that the number of jump points is constant. Let $(v^{\varepsilon_{j},\tau})$ be a subsequence of $(v^{\varepsilon,\tau})$ such that the number of jump points of $(v^{\varepsilon_{j},\tau})$ is constant and equal to $N({\tau})$ for some $N({\tau}) \in \mathbb{N}$. For each $j$, let $\{t_n(\varepsilon_j)\}_{n=1}^{N(\tau)}$ denote the jump points of $v^{\varepsilon_j,\tau}$. By passing to a further subsequence if necessary, we may assume that
$t_n(\varepsilon_j)\to t_n$ as $j\to\infty$ for all $n=1,\ldots,N(\tau)$, $v^{\varepsilon_j,\tau}\to v^{\tau}$ as $j\to\infty$ pointwise almost everywhere for some $v^{\tau}\in BV((0,1);\{-1,1\})$, and the jump set of $v^{\tau}$ is contained in $\{t_1,\ldots,t_{N(\tau)}\}$. 
Analogously, after passing to a subsequence if necessary, we may assume that $v^{\tau}\to v$ pointwise almost everywhere as $\tau\to 0$, for some $v\in BV((0,1);\{-1,1\})$. Hence, after passing to a further subsequence and letting first $\varepsilon\to0$ and then $\tau\to0$, we may assume that $v^{\varepsilon,\tau}$ converges pointwise almost everywhere \ to some $v \in BV((0,1);\{-1,1\})$. In particular, both $v^{\varepsilon,\tau}$ and $v^{\tau}$ are uniformly bounded in $L^{\infty}(0,1)$. Therefore, by the dominated convergence theorem, almost everywhere convergence implies convergence in $L^{2}(0,1)$. Finally, by a diagonal argument, we may extract a sequence $\varepsilon_i\to0$ and a subsequence $\tau_k\to0$ such that, for every $k\in\N$, $v^{\varepsilon_i,\tau_k}\to v^{\tau_k}$ as $i\to\infty$, while $v^{\tau_k}\to v$ as $k\to\infty$. Moreover, the sequence $\varepsilon_i$ is independent of $k$. We now show that $u_{\varepsilon_i}\to v$ as $i\to\infty$ in $L^{2}(0,1)$.

We now claim that for every $\tau \in (0,1/8)$, the following estimate holds:
\begin{equation}\label{36tau}
    \int_{0}^{1}\left|u_{\varepsilon}(t)-v^{\varepsilon,\tau}(t)\right|^2dt\le36\tau+o(1),
\end{equation}
as $\varepsilon\to0$. Indeed, we decompose 
\begin{equation*}
    \int_{0}^{1}\left|u_{\varepsilon}(t)-v^{\varepsilon,\tau}(t)\right|^2dt=\int_{\{\left|u_{\varepsilon}\right|>1+\tau\}}\left|u_{\varepsilon}(t)-v^{\varepsilon,\tau}(t)\right|^2dt+\int_{\{\left|u_{\varepsilon}\right|\le1+\tau\}}\left|u_{\varepsilon}(t)-v^{\varepsilon,\tau}(t)\right|^2dt.
\end{equation*}
We claim that 
\[
\lim_{\varepsilon\to0}\int_{\{\left|u_{\varepsilon}\right|>1+\tau\}}\left|u_{\varepsilon}(t)-v^{\varepsilon,\tau}(t)\right|^2dt=0.
\]
By Hypothesis {\rm (A2)}, we obtain the following upper bound:
\begin{align*}
    \int_{\{\left|u_{\varepsilon}\right|>1+\tau\}}\left|u_{\varepsilon}(t)-v^{\varepsilon,\tau}(t)\right|^2dt\le2\int_{\{\left|u_{\varepsilon}\right|>1+\tau\}}(u^{2}_{\varepsilon}(t)+1)dt\\
\le\frac{2}{c_{1}}\int_{0}^{1}W(u_{\varepsilon}(t))dt+2\left(\frac{c_{2}}{c_{1}}+1\right)|\{\left|u_{\varepsilon}\right|>1+\tau\}|.
\end{align*}
Consequently, by Proposition \ref{convergence in measure}
\[
\lim_{\varepsilon\to0}\int_{\{\left|u_{\varepsilon}\right|>1+\tau\}}\left|u_{\varepsilon}(t)-v^{\varepsilon,\tau}(t)\right|^2dt\le\lim_{\varepsilon\to0}\left(\frac{2C\varepsilon|\log{\varepsilon}|}{c_{1}}+2\left(\frac{c_{2}}{c_{1}}+1\right)|\{\left|u_{\varepsilon}\right|>1+\tau\}|\right)=0.
\]
Now we will show that
\[
\int_{\{\left|u_{\varepsilon}\right|\le1+\tau\}}\left|u_{\varepsilon}(t)-v^{\varepsilon,\tau}(t)\right|^2dt\le 36\tau+o(1),
\]
as $\varepsilon\to0$. We define the sets $J_1(\varepsilon)$, $J_2(\varepsilon)$ and $J_3(\varepsilon)$ as follows:
\begin{equation*}
    J_{1}(\varepsilon):=\bigcup_{k\in P_{1}(\varepsilon)}I_{k}(\varepsilon)\bigcap \{\left|u_{\varepsilon}\right|\le{1+\tau\}},\,\,J_{2}(\varepsilon):=\bigcup_{k\in P_{2}(\varepsilon)}I_{k}(\varepsilon)\bigcap \{\left|u_{\varepsilon}\right|\le{1+\tau\}},
\end{equation*}
\begin{equation*}
    J_{3}(\varepsilon):=I_{n(\varepsilon)+1}\bigcup_{k\in P_{3}(\varepsilon)}I_{k}(\varepsilon)\bigcap \{\left|u_{\varepsilon}\right|\le{1+\tau\}}.
\end{equation*}
Then, we have 
\begin{multline*}
    \int_{\{\left|u_{\varepsilon}\right|\le1+\tau\}}\left|u_{\varepsilon}(t)-v^{\varepsilon,\tau}(t)\right|^2dt= \int_{J_{1}(\varepsilon)}\left|u_{\varepsilon}(t)+1\right|^2dt+\int_{J_{2}(\varepsilon)}\left|u_{\varepsilon}(t)-1\right|^2dt\\+\int_{J_{3}(\varepsilon)}\left|u_{\varepsilon}(t)-1\right|^2dt\le\int_{J_{1}(\varepsilon)}\left|u_{\varepsilon}(t)+1\right|^2dt+\int_{J_{2}(\varepsilon)}\left|u_{\varepsilon}(t)-1\right|^2dt\\+(2+\tau)^{2}(\#P_{3}(\varepsilon)+1)\theta(\varepsilon).
\end{multline*}
Then, by the uniform bound on $\#P_3(\varepsilon)$, we obtain
\begin{equation*}
\int_{\{\left|u_{\varepsilon}\right|\le1+\tau\}}
\left|u_{\varepsilon}(t)-v^{\varepsilon,\tau}(t)\right|^2dt
\le\int_{J_{1}(\varepsilon)}
\left|u_{\varepsilon}(t)+1\right|^2dt+\int_{J_{2}(\varepsilon)}
\left|u_{\varepsilon}(t)-1\right|^2dt+o(1),
\end{equation*}
as $\varepsilon\to0$. We claim that 
\begin{align*}
     \int_{J_{1}(\varepsilon)}\left|u_{\varepsilon}(t)+1\right|^2dt\le18\tau,\,\int_{J_{2}(\varepsilon)}\left|u_{\varepsilon}(t)-1\right|^2dt\le18\tau.
\end{align*}
We give the proof for the first estimate; the second one follows similarly
 \begin{multline*}
     \int_{J_{1}(\varepsilon)}\left|u_{\varepsilon}(t)+1\right|^2dt=\int_{\cup_{k\in P_{1}}A^{c}_k(\varepsilon)\cap\{\left|u_{\varepsilon}\right|\le\tau+1\}}\left|u_{\varepsilon}(t)+1\right|^2dt\\+\int_{\cup_{k\in P_{1}(\varepsilon)}A_k(\varepsilon)\cap\{\left|u_{\varepsilon}\right|\le\tau+1\}}\left|u_{\varepsilon}(t)+1\right|^2dt\le2\tau(2+\tau)^{2}+\tau^{2}\le18\tau.
 \end{multline*}
 The above estimates yield the claimed result. Then, by \eqref{36tau} and the convergence of $v^{\varepsilon_i,\tau_k}\to v^{\tau_k}$ in $L^{2}(0,1)$ as $i\to\infty$, we obtain
\begin{multline*} \limsup_{i\to\infty}||u_{\varepsilon_{i}}-v|| \le\limsup_{i\to\infty}||u_{\varepsilon_{i}}-v^{\varepsilon_{i},\tau_{k}}||_{L^{2}(0,1)}+\limsup_{i\to\infty}||v^{\varepsilon_{i},\tau_{k}}-v^{\tau_{k}}||_{L^{2}(0,1)}\\+||v^{\tau_{k}}-v||_{L^{2}(0,1)}\le6\sqrt{\tau_{k}}+||v^{\tau_{k}}-v||_{L^{2}(0,1)} 
\end{multline*}
for every $k \in \mathbb{N}$. By arbitrariness of $k$, letting $k \to \infty$ yields $\lim_{i\to\infty} \|u_{\varepsilon_i}-v\|_{L^2(0,1)} = 0,$ that is, $u_{\varepsilon_i} \to v$ strongly in $L^2(0,1)$. Finally, passing to a further subsequence, we may assume that the convergence holds almost everywhere in $(0,1)$. This concludes the proof of the compactness result.
\end{proof}

\section{Proof of the Gamma-convergence result}

Theorem \ref{main}(ii) can be restated as follows.

\begin{theorem}\label{sub-main}
Let $W:\R\to[0,+\infty)$ be a double-well potential that satisfies assumptions {\rm (A1)} and {\rm (A2)}, and let $a:\R^2\to[\alpha_{a},\beta_{a}]$ satisfy assumption {\rm (A3)}. Then the following holds: 

\smallskip
{\rm(i) (liminf inequality)} every sequence $(u_{\varepsilon})_{\varepsilon}$ converging to $u$ in $L^2(0,1)$ satisfies 
$\liminf\limits_{\varepsilon\to0}F_{\varepsilon}(u_{\varepsilon})\geq F(u)$;

\smallskip
{\rm(ii) (limsup inequality)} for every $u\in\BV\left((0,1);\{-1,1\}\right)$ there exist $(u_{\varepsilon})_{\varepsilon}$ converging to $u$ in $L^2(0,1)$ such that 
 $\lim\limits_{\varepsilon\to 0}F_{\varepsilon}(u_{\varepsilon})=F(u)$.
\end{theorem} 

We separately prove claim (i) and claim (ii) of Theorem \ref{sub-main} in the following two sections.
In the proofs, we will use the localized functionals defined for every $\varepsilon > 0$, 
$u \in H^{1/2}(0,1)$ and for every open interval $I \subset (0,1)$, by setting
\begin{equation*}
F_{\varepsilon}(u,I) :=
\frac{1}{\varepsilon\left|\log\varepsilon\right|}\int_{I} W(u)\,dt+\frac{1}{\left|\log\varepsilon\right|}
\int_{I}\int_{I}a\Big(\frac{s}{\delta(\varepsilon)},\frac{t}{\delta(\varepsilon)}
    \Big)
\frac{\left|u(s)-u(t)\right|^2}{\left|s-t\right|^2}
dsdt.
\end{equation*}

\subsection{Proof of the liminf inequality}
This section is devoted to the proof of Theorem \ref{sub-main}(i). The argument relies on a localization of the energy near the jump points of the limit function and on separate estimates of the energy contributions arising at different scales.

\bigskip
Let $(u_{\varepsilon})_{\varepsilon}\in L^2(0,1)$ be such that $u_{\varepsilon}\to u$ as $\varepsilon\to 0$ in $L^2(0,1)$. 
We can assume that $\liminf_{\varepsilon\to0}F_{\varepsilon}(u_{\varepsilon})\le C<\infty$ and thus by Proposition \ref{compactness} we have that $u\in\BV((0,1);\{-1,1\})$. Otherwise, there is nothing to prove. Let $(\varepsilon_{i})_{i\ge1}$ be a sequence that converges to $0$ such that  
$\liminf_{\varepsilon\to0}F_{\varepsilon}(u_{\varepsilon})=\lim_{i\to\infty}F_{\varepsilon_{i}}(u_{\varepsilon_{i}})$. 
 Then, by the compactness result (possibly extracting a further subsequence), we can assume that $u_{\varepsilon_{i}}\to u$ as $i\to\infty$ almost everywhere in $(0,1)$.
 We will work with that subsequence and we write $u_{\varepsilon_{i}}=u_{\varepsilon}$. 
 
 Let $t_n$, with $n\in\{1,\ldots,\#S(u)\}$, be the jump points of $u$ and let $\rho>0$ be such that $\#\{(t_n-\rho,t_n+\rho)\cap\{t_m: 1\le m\le \#S(u)\}=1$ and $(t_n-\rho,t_n+\rho)
 \subset(0,1)$ for any $n\in\{1,...,\#S(u)\}$. Then, we have 
 \begin{equation*}
     \lim_{\varepsilon\to0}F_{\varepsilon}(u_{\varepsilon})\ge\sum_{n=1}^{\#S(u)}\liminf_{\varepsilon\to0}F_{\varepsilon}(u_{\varepsilon},(t_n-\rho,t_n+\rho)).
 \end{equation*}
 We claim that for every $\tau\in(0,\frac{1}{8})$, $\theta'\in(\lambda,+\infty)$, $\theta''\in(-\infty,\lambda)$ and every $n\in\{1,...,\#S(u)\}$ we have
 \begin{equation}\label{local liminf}
     \liminf_{\varepsilon\to0}F_{\varepsilon}(u_{\varepsilon},(t_n-\rho,t_n+\rho))\ge8(1-\tau)^{2}((1-\theta')a_{\min}+\theta''(\bar{a}-6\beta_{a}\tau)).
 \end{equation}
This will prove claim (i) of the theorem. Indeed, note that the right-hand side is independent of $\tau$, so we may pass to the limit as $\tau\to0$. Then we obtain 
 \[
  \liminf_{\varepsilon\to0}F_{\varepsilon}(u_{\varepsilon},(t_n-\rho,t_n+\rho))\ge8((1-\theta')a_{\min}+\theta''\bar{a}).
 \]
 for every $\theta'\in(\lambda,+\infty)$ and every $\theta''\in(-\infty,\lambda)$. Again, passing to the limit as $\theta',\theta''\to\lambda$ we obtain that 
 \[
  \liminf_{\varepsilon\to0}F_{\varepsilon}(u_{\varepsilon},(t_n-\rho,t_n+\rho))\ge8((1-\lambda)a_{\min}+\lambda\bar{a}).
 \]
 Finally, summing over $n$ yields the liminf inequality.

 By the above argument, we can assume that $u$ has only one jump point $t_1$. Since the proof of \eqref{local liminf} is rather technical and lengthy, we divide it into two steps.

\def\thet{\theta(\epsilon)}

\smallskip
\noindent\textbf{Step 1.} 
  In this step we prove that there exists a subinterval $I^\e_{1}\subset (t_1-\rho,t_1+\rho)$ such that 
  \begin{equation}\label{first-lb}    \liminf_{\varepsilon\to0}F_{\varepsilon}(u_{\varepsilon},I^\e_1)\ge 8(1-\tau)^{2}(1-\theta')a_{\min}.
 \end{equation}
 for every $\theta'\in(\lambda,\infty)$ and every $\tau\in(0,\frac{1}{8}).$
 This inequality is trivial if $\lambda=1$ or $\theta'\ge 1$. Hence, we can suppose that $\lambda=
\lambda_0<1$ and $\lambda<\theta'<1$. Following the notation of Section \ref{discr}, we consider a family $(\theta'(\varepsilon))_{\e}$ that satisfies \eqref{theta-eps} with $\theta$ replaced by $\theta'$.
Correspondingly, we consider a partition of $(t_1-\rho,t_1+\rho)$ as follows:
 \begin{equation*}
     (t_1-\rho,t_1+\rho)=\bigcup_{k=-n(\varepsilon)-1}^{n(\varepsilon)}I'_{k}(\varepsilon),
 \end{equation*}
 where $n(\varepsilon)=\lfloor\frac{\rho}{\theta'(\e) }\rfloor$ and
 \[
 I'_{k}(\varepsilon)=
 \begin{cases}
   (t_{1}-\rho,t_{1}-n({\varepsilon})\theta'(\e) ] &
   \hbox{if}\,\,k=-n(\varepsilon)-1\\
   [k\theta'(\e) ,(k+1)\theta'(\e) ] & \hbox{if}\,\,k\in\{-n(\varepsilon),...,n(\varepsilon)-1\}\\
   [n(\varepsilon)\theta'(\e) +t_{1},t_{1}+\rho) & \hbox{if}\,\,k=n(\varepsilon).
 \end{cases}
 \]
We define the corresponding sets $A'_{k}(\e)$, $B'_{k}(\e)$, their measures $a'_{k}(\e),b'_{k}(\e)$, and the associated index sets $P'_1(\e)$, $P'_2(\e)$, $P'_3(\e)$, and $\bar P'_3(\e)$. Using the same argument as in the proof of compactness, we can show that the following hold.\\
(a) For every $\tau\in(0,\frac{1}{8})$, there is $\varepsilon(\tau)>0$ such that for all $(0,\varepsilon(\tau))$
\[
P'_{1}(\varepsilon)\cup P'_{2}(\varepsilon)\cup P'_{3}(\varepsilon)=\{-n(\varepsilon),...,n(\varepsilon)-1\}.
\]

\noindent(b) $\#\big(P'_{3}(\varepsilon)\cup \bar{P}'_{3}(\varepsilon)\big)$ is uniformly bounded, independent of $\varepsilon$ and $\tau$.

\smallskip

\noindent (c) There exist $\bar{\varepsilon}(\tau)>0$ such that
$P'_{3}(\varepsilon)\cup \bar{P}'_{3}(\varepsilon)\neq\emptyset$,
for every $\varepsilon\in(0,\bar{\varepsilon}(\tau))$. 

\bigskip

These properties allow us to identify sufficiently many intervals in which $u_\varepsilon$ transitions between the two wells, yielding the desired lower bound. It suffices to prove part (c), as the remaining parts follow as in Proposition \ref{3.2}.

 Assume by contradiction that there is a subsequence $(\varepsilon_{i_{p}})_{p\ge1}$ of $(\varepsilon_{i})_{i\ge1}$ and $N\in\N$ such that for all $p\ge N$ we get $P_{3}(\varepsilon_{i_{p}})\cup \bar{P}_{3}(\varepsilon_{i_{p}})=\emptyset$ and we can choose $N$ large enough such that $\varepsilon_{i_{p}}\in(0,\varepsilon(\tau))$ for all $p\ge N$. This implies that $P_{1}(\varepsilon_{i_{p}})=\emptyset$ or $P_{2}(\varepsilon_{i_{p}})=\emptyset$ for all $p\ge N$. Suppose $P_{2}(\varepsilon_{i_{p}})=\emptyset$ for all $p\ge N$ and we know that $u_{\varepsilon_{i_{p}}}|_{(t_{1}-\rho,t_{1}+\rho)}\to u|_{(t_{1}-\rho,t_{1}+\rho)}$ as $p\to\infty$. This implies that on the one hand 
\begin{equation}\label{4.2}
    |\{t\in(t_{1}-n(\varepsilon)\theta'(\e) ,t_{1}+n(\varepsilon)\theta'(\e) ):\,u_{\varepsilon}(t)\le\tau-1\}|\ge2n(\varepsilon)\theta'(\e) (1-2\tau),
\end{equation}
and, on the other hand the measure 
$
|\{t\in[t_{1}-n(\varepsilon_{i_{p}}){\theta'(\e)a(\varepsilon_{i_{p}})},t_{1}+n(\varepsilon_{i_{p}}){\theta'(\e)a(\varepsilon_{i_{p}})}]:\,u_{\varepsilon_{i_{p}}}(t)\le\tau-1\}|$
tends to
$\rho$ as $p\to\infty$. 
Indeed, we can show that
\begin{multline*}
    \lim_{\varepsilon\to0}|\{t\in(t_{1}-\rho,t_{1}+\rho);\,u_{\varepsilon}(t)\le\tau-1\}|=\lim_{\varepsilon\to0}|\{t\in(t_{1}-\rho,t_{1}+\rho);\,u_{\varepsilon}(t)\ge1-\tau\}|=\rho.
\end{multline*}
We can assume the following. 
\begin{equation*}
    u|_{(t_{1}-\rho,t_{1}+\rho)}(t)=
    \begin{cases}
     -1 & t\in(t_{1}-\rho,t_{1})\\
     1& t\in(t_{1},t_{1}+\rho).
    \end{cases}
\end{equation*} 
Otherwise, we argue analogously after replacing $u$ with $-u$. Let $\kappa>0$ be fixed. Then, according to Egorov's theorem, there exists $G(\kappa)\subset(t_{1}-\rho,t_{1})$ such that $u_{\varepsilon}|_{G(\kappa)}\to-1$ uniformly in $G(\kappa)$ and $|(t_{1}-\rho,t_{1})\setminus G(\kappa)|<\kappa$. This implies there is $\varepsilon(\kappa)>0$ such that 
$
|\{t\in(t_{1}-\rho,t_{1}):\,u_{\varepsilon}(t)\le\tau-1\}|\ge|G(\kappa)|\ge\rho-\kappa$, 
for every $\e\in(0,\e(\kappa))$. Then we get 
\[
\liminf_{\varepsilon\to0}|\{t\in(t_{1}-\rho,t_{1}+\rho);\,u_{\varepsilon}(t)\le\tau-1\}|\ge\rho-\kappa.
\]
Since the left-hand side is independent of $\kappa$, we may let $\kappa\to0$. Passing to the limit as $\kappa\to0$, we obtain 
\[
\liminf_{\varepsilon\to0}|\{t\in(t_{1}-\rho,t_{1}+\rho);\,u_{\varepsilon}(t)\le\tau-1\}|\ge\rho.
\]
Similarly, we can show that $\liminf_{\varepsilon\to0}|\{t\in(t_{1}-\rho,t_{1}+\rho);\,u_{\varepsilon}(t)\ge1-\tau\}|\ge\rho.$

On the other hand,
\begin{multline*}
    \limsup_{\varepsilon\to0}|\{t\in(t_{1}-\rho,t_{1}+\rho);\,u_{\varepsilon}(t)\le\tau-1\}|\\+\limsup_{\varepsilon\to0}|\{t\in(t_{1}-\rho,t_{1}+\rho);\,u_{\varepsilon}(t)\ge1-\tau\}|\le 2\rho.
\end{multline*}
Thus, we can conclude 
\begin{multline*}
    \lim_{\varepsilon\to0}|\{t\in(t_{1}-\rho,t_{1}+\rho);\,u_{\varepsilon}(t)\le\tau-1\}|\\=\lim_{\varepsilon\to0}|\{t\in(t_{1}-\rho,t_{1}+\rho);\,u_{\varepsilon}(t)\ge1-\tau\}|=\rho.
\end{multline*}
Consequently, by \eqref{4.2} we get
\begin{multline*}
     \rho=\lim_{\varepsilon\to0}|\{t\in(t_{1}-\rho,t_{1}+\rho);\,u_{\varepsilon}(t)\le\tau-1\}|\\=\lim_{\varepsilon\to0}|\{t\in[t_{1}-n(\varepsilon)\theta'(\e) ,t_{1}+n(\varepsilon)\theta'(\e) ]:\,u_{\varepsilon}(t)\le\tau-1\}|\ge2\rho(1-2\tau)\ge\frac{3\rho}{2},
\end{multline*}
which, since $1-2\tau
\ge\frac34$, contradicts the nonexistence of $\bar{\varepsilon}(\tau)$.

We now define a suitable subinterval $I^\varepsilon_1$ by
    \[
    I^{\e}_{1}=\begin{cases}
    [k\theta'(\e) ,(k+1)\theta'(\e) ] & \hbox{if}\,\, k\in P'_{3}(\varepsilon)\\
    [k\theta'(\e) ,(k+2)\theta'(\e) ] & \hbox{if}\,\, k\in \bar{P}'_{3}(\varepsilon).
    \end{cases}
    \]
and show that
$\liminf\limits_{\varepsilon\to0}F_{\varepsilon}(u_{\varepsilon},I^{\e}_{1})\ge8(1-\tau)^{2}(1-\theta')a_{\min}$. Indeed, since $\lambda<\theta'$, we have $|s-t|/\delta(\e)\le2\theta'(\e)/\delta(\e)\to0\,\,\text{as}\,\,\e\to0$ for every $s,t\in I^{\e}_{1}$. Therefore, 
$
a\big(\frac{s}{\delta(\e) },\frac{t}{\delta(\e) }
    \big)=a\big(\frac{t}{\delta(\e) },\frac{t}{\delta(\e) }\big)+o(1)\ge a_{\min}+o(1)
$ 
as $\e\to 0$, by the uniform continuity of $a$. Consequently, we obtain
\begin{multline*}
F_{\varepsilon}(u_{\varepsilon},I^{\e}_{1})=\frac{1}{\varepsilon\left|\log\varepsilon\right|}\int_{I^{\e}_{1}} W(u_{\varepsilon})\,dt\\+\frac{1}{\left|\log\varepsilon\right|}
\int_{I^{\e}_{1}}\int_{I^{\e}_{1}}a\Big(\frac{s}{\delta(\e) },\frac{t}{\delta(\e) }
    \Big)
\frac{\left|u_{\varepsilon}(s)-u_{\varepsilon}(t)\right|^2}{\left|s-t\right|^2}\
ds\,dt
\\
\ge \frac{1}{\varepsilon\left|\log\varepsilon
\right|}\int_{I^{\e}_{1}} W(u_{\varepsilon})dt+\frac{a_{\min}+o(1)}{\left|\log\varepsilon\right|}
\int_{I^{\e}_{1}}\int_{I^{\e}_{1}}
\frac{\left|u_{\varepsilon}(s)-u_{\varepsilon}(t)\right|^2}{\left|s-t\right|^2}
dsdt  
\end{multline*}
    as $\varepsilon\to0$. Applying Lemma \ref{lower bound lemma}, we obtain
\begin{multline*}
F_{\varepsilon}(u_{\varepsilon},I^{\e}_{1}) \ge
\frac{8(1-\tau)^{2}(a_{\min}+o(1))}{|\log{\varepsilon}|}
\Big( 2\log{\tau} + |\log\varepsilon|
+ \log|I^{\e}_{1}| 
+ \log\frac{\sigma(\tau)}{8\beta_{a}(1-\tau)^{2}} \Big)
\end{multline*}
    as $\varepsilon\to0$, from which we have 
    \eqref{first-lb}. 

\medskip
\noindent \textbf{Step 2.}  In this step, we prove that 
  \begin{equation}\label{second-lb}    \liminf_{\varepsilon\to0}F_{\varepsilon}(u_{\varepsilon},(t_{1}-\rho,t_{1}+\rho)\setminus I^\e_1)\ge 8\theta''(1-\tau)^{2}(\bar{a}-6\tau\beta_a).
 \end{equation}
 for every $\theta''\in(-\infty,\lambda)$ and every $\tau\in(0,\frac{1}{8})$.
 This inequality is trivial if $\lambda=0$ or $\theta''\le 0$. Hence, we can suppose that $\lambda=
\lambda_0>0$ and $0<\theta''<\lambda$. Similarly to Step~1, we consider a family $(\theta''(\e))_{\e}$ that satisfies \eqref{theta-eps} with $\theta$ replaced by $\theta''$. Moreover, we may assume without loss of generality that $\theta''(\varepsilon)$ is an integer multiple of $\delta(\varepsilon)$ for some $N(\theta''(\e))$, e.g. $\theta''(\varepsilon) = \left\lfloor \varepsilon^{\theta''}/{\delta(\varepsilon)} \right\rfloor \delta(\varepsilon)$. We now repeat the construction from Step 1 at the scale $\theta''(\varepsilon)$. More precisely, we consider the corresponding partition of $(t_1-\rho,t_1+\rho)$ obtained by replacing $\theta'(\varepsilon)$ with $\theta''(\varepsilon)$, and we define the associated sets and index families as in Step 1. Moreover, properties (a)–-(c) in Step 1 still hold. Recall that by Step 1 there exists $I^{\e}_{1}\in(t_{1}-\rho,t_{1}+\rho)$ such that
 \begin{equation*}
\liminf_{\varepsilon\to0}F_{\varepsilon}(u_{\varepsilon},I^{\e}_{1})\ge8(1-\tau)^{2}(1-\theta')a_{\min}
 \end{equation*}
 for every $\lambda<\theta''<1$. We have $\theta''<\lambda<\theta'$ and thus $\theta'(\e)\ll\theta''(\e)$. Since the partition associated with $\theta''(\e)$ is obtained at a larger scale, each interval in the $\theta'(\e)$-partition is contained in the union of at most two consecutive intervals of the $\theta''(\e)$-partition. In particular, 
we have
$I^\e_{1} \subset I''_{k_1}(\e)\cup I''_{k_{1}+1}(\e)$
for some indices $k_{1}=k_{1}(\e)$. Consequently, the lower bound obtained on $I^\e_{1}$ in Step 1 can be transferred to the union of these subintervals, and the same argument applies at the $\theta''(\e)$-scale. This implies that
 \begin{multline*}
     \liminf_{\e\to0}F_{\e}(u_{\e},(t_1-\rho,t_1+\rho))\ge \liminf_{\e\to0}F_{\e}(u_{\e},I''_{k_1}(\e)\cup I''_{k_{1}+1}(\e))\\+\liminf_{\e\to0}\frac{1}{\left|\log{\e}\right|}\int_{I''(\varepsilon)}\int_{I''(\e)}a\big(\frac{s}{\delta(\e) },\frac{t}{\delta(\e) }\big)\frac{\left|u_{\e}(s)-u_{\e}(t)\right|^2}{\left|s-t\right|^2}ds\,dt\\\ge8(1-\tau)^{2}(1-\theta')a_{\min}
     \\+\liminf_{\e\to0}\frac{1}{\left|\log{\e}\right|}\int_{I''(\varepsilon)}\int_{I''(\e)}a\big(\frac{s}{\delta(\e) },\frac{t}{\delta(\e) }\big)\frac{\left|u_{\e}(s)-u_{\e}(t)\right|^2}{\left|s-t\right|^2}ds\,dt,
 \end{multline*}
 where $P''_{0}(\e)=P''_{3}(\e)\cup\{k_{1},k_{1}+1\}$ and $I''(\e):=\bigcup_{k\notin P''_{0}(\e)}I''_{k}(\e)$. Note that $I''(\e)\in(t_1-\rho,t_1+\rho)\setminus I^{\e}_1$ and hence,
\begin{multline*}\liminf_{\varepsilon\to0}F_{\varepsilon}(u_{\varepsilon},(t_{1}-\rho,t_{1}+\rho)\setminus I^\e_1)\\\ge\liminf_{\e\to0}\frac{1}{\left|\log{\e}\right|}\int_{I''(\varepsilon)}\int_{I''(\e)}a\big(\frac{s}{\delta(\e) },\frac{t}{\delta(\e) }\big)\frac{\left|u_{\e}(s)-u_{\e}(t)\right|^2}{\left|s-t\right|^2}ds\,dt.
\end{multline*}

 We now show that 
 \begin{equation}\label{liminf average term}
     \liminf_{\varepsilon\to0}\frac{1}{\left|\log{\varepsilon}\right|}\int_{I''(\e)}\int_{I''(\e)}a\Big(\frac{s}{\delta(\e) },\frac{t}{\delta(\e) }\Big)\frac{\left|u_{\varepsilon}(s)-u_{\varepsilon}(t)\right|^2}{\left|s-t\right|^2}ds\,dt\ge8\theta''(1-\tau)^{2}(\bar{a}-6\beta_{a}\tau),
 \end{equation}
 which proves \eqref{second-lb}. 
 
 We estimate the integral in \eqref{liminf average term} as follows.
 \begin{multline*}    \int_{I''(\e)}\int_{I''(\e)}a\Big(\frac{s}{\delta(\e) },\frac{t}{\delta(\e) }\Big)\frac{\left|u_{\varepsilon}(s)-u_{\varepsilon}(t)\right|^2}{\left|s-t\right|^2}ds\,dt\\\ge2\sum_{k\in\widetilde{P}''_{1}(\varepsilon)}\sum_{m\in\widetilde{P}''_{2}(\varepsilon)}\int_{I''_{k}(\varepsilon)}\int_{I''_{m}(\varepsilon)}a\Big(\frac{s}{\delta(\e) },\frac{t}{\delta(\e) }\Big)\frac{\left|u_{\varepsilon}(s)-u_{\varepsilon}(t)\right|^2}{\left|s-t\right|^2}ds\,dt\\\ge\sum_{k\in\widetilde{P}''_{1}(\varepsilon)}\sum_{m\in\widetilde{P}''_{2}(\varepsilon)}\frac{8(1-\tau)^{2}}{(|m-k|+1)^{2}(\theta''(\e))^{2}}\int_{A''_{k}(\varepsilon)}\int_{B''_{m}(\varepsilon)}a\Big(\frac{s}{\delta(\varepsilon)},\frac{t}{\delta(\e)}\Big)dtds, 
 \end{multline*}
 where
$\widetilde{P}''_1(\varepsilon)=P''_{1}(\varepsilon)\setminus P''_{0}(\varepsilon)$, $\widetilde{P}''_{2}(\varepsilon)=P''_{2}(\varepsilon)\setminus P''_{0}(\varepsilon)$. Moreover, we have
\begin{multline*}
    \int_{A''_{k}(\varepsilon)}\int_{B''_{m}(\varepsilon)}a\Big(\frac{s}{\delta(\varepsilon)},\frac{t}{\delta(\varepsilon)}\Big)dtds=\int_{I''_{k}(\varepsilon)}\int_{I''_{m}(\varepsilon)}a\Big(\frac{s}{\delta(\varepsilon)},\frac{t}{\delta(\varepsilon)}\Big)dtds\\- \int_{(A''_{k}(\varepsilon))^{c}}\int_{B''_{m}(\varepsilon)}a\Big(\frac{s}{\delta(\varepsilon)},\frac{t}{\delta(\varepsilon)}\Big)dtds- \int_{A''_{k}(\varepsilon)}\int_{(B''_{m}(\varepsilon))^{c}}a\Big(\frac{s}{\delta(\varepsilon)},\frac{t}{\delta(\varepsilon)}\Big)dtds\\-\int_{(A''_{k}(\varepsilon))^{c}}\int_{(B''_{m}(\varepsilon))^{c}}a\Big(\frac{s}{\delta(\varepsilon)},\frac{t}{\delta(\varepsilon)}\Big)dtds.
\end{multline*}
Since \(k\in\widetilde{P}''_1(\varepsilon)\) and
\(m\in\widetilde{P}''_2(\varepsilon)\), each of the last three terms can be estimated from above by
\(2\tau\beta_a(\theta''(\varepsilon))^2\). Hence,
\begin{multline*}
    \frac{1}{(\theta''(\varepsilon))^{2}}\int_{A''_{k}(\varepsilon)}\int_{B''_{m}(\varepsilon)}a\Big(\frac{s}{\delta(\varepsilon)},\frac{t}{\delta(\varepsilon)}\Big)dtds
    \\
    \ge\frac{1}{(\theta''(\varepsilon))^{2}}\int_{I''_{k}(\varepsilon)}\int_{I''_{m}(\varepsilon)}a\Big(\frac{s}{\delta(\varepsilon)},\frac{t}{\delta(\varepsilon)}\Big)dtds-6\tau\beta_{a}
    \\
    =\frac{\delta^{2}(\varepsilon)}{(\theta''(\varepsilon))^{2}}\int_{\frac{k(\theta''(\varepsilon))}{\delta(\varepsilon)}}^{\frac{(k+1)(\theta''(\varepsilon))}{\delta(\varepsilon)}}\int_{\frac{m(\theta''(\varepsilon))}{\delta(\varepsilon)}}^{\frac{(m+1)(\theta''(\varepsilon))}{\delta(\varepsilon)}}a\Big(s,t\Big)dtds-6\tau\beta_{a}
    \\
    =\frac{1}{N^2(\theta''(\e))}\int_{1}^{N(\theta''(\e))}\int_{1}^{N(\theta''(\e))}a\Big(s,t\Big)dtds-6\tau\beta_{a}
    =\bar{a}-6\tau\beta_{a}.
\end{multline*} 
Combining the previous estimate with the above lower bound, we obtain
\begin{multline*}
    \liminf_{\e\to0}\frac{1}{\left|\log{\varepsilon}\right|}\int_{I''(\e)}\int_{I''(\e)}a\Big(\frac{s}{\delta(\e) },\frac{t}{\delta(\e) }\Big)\frac{\left|u_{\varepsilon}(s)-u_{\varepsilon}(t)\right|^2}{\left|s-t\right|^2}ds\,dt\\\ge\liminf_{\e\to0}\frac{1}{\left|\log{\varepsilon}\right|}\sum_{k\in\widetilde{P}''_{1}(\varepsilon)}\sum_{m\in\widetilde{P}''_{2}(\varepsilon)}\frac{8(1-\tau)^{2}(\bar{a}-6\beta_{a}\tau)}{(|m-k|+1)^{2}}.
\end{multline*}

We claim that
\begin{equation}\label{harmonic sum l.b}
    \liminf_{\varepsilon\to0}\frac{1}{\left|\log{\varepsilon}\right|}\sum_{k\in\widetilde{P}''_{1}(\varepsilon)}\sum_{m\in\widetilde{P}''_{2}(\varepsilon)}\frac{1}{(|m-k|+1)^2}\ge\theta''.
\end{equation}
To prove \eqref{harmonic sum l.b}, we make use of the following discrete rearrangement lemma for nonincreasing kernels. This result allows us to conclude the proof of \eqref{liminf average term}. The lemma shows that, among all configurations of two disjoint subsets with prescribed cardinalities, the minimal interaction is achieved when the sets are arranged in opposite extremal positions.
 \begin{lemma}\label{rearrangement lemma}
Let $\phi:(0,+\infty)\to(0,+\infty)$ be a nonincreasing function and $N_{1},\,N_{2}\in\mathbb{Z}$. Then, we have 
\begin{equation}\label{rearrangement inequality}
    \sum_{k\in P_{1}}\sum_{m\in P_{2}}\phi(|k-m|)\ge\sum_{N_{1}}^{k=k_{0}}\sum_{m=m_{0}}^{N_{2}}\phi(|k-m|)
\end{equation}
where $P_{1},\,P_{2}\subset\{N_{1},...,N_{2}\}$ are disjoint sets, $k_{0}=N_{1}+(\#P_{1}-1)$ and $m_{0}=N_{2}-(\#P_{2}-1)$.
 \end{lemma}
 \begin{proof}
We will prove this by induction on $n=\#(P_{1}\cup P_{2})$. Suppose $n=\#(P_{1}\cup P_{2})=2$, then there is nothing to prove. Suppose that the claim is true for $n=\#(P_{1}\cup P_{2})-1$, then we will prove it for $n=\#(P_{1}\cup P_{2})$. Let $\bar{k}\in P_{1}\cup P_{2}$ be a minimal element. Since inequality \eqref{rearrangement inequality} is invariant under interchanging the sets $P_{1}$ and $P_{2}$, we may assume without loss of generality that $\bar{k}\in P_{1}$. Then, we have
    
\begin{align*}
\sum_{k\in P_{1}}\sum_{m\in P_{2}}\phi(|k-m|)=\sum_{k\in P_{1}\setminus\{\bar{k}\}}\sum_{m\in P_{2}}\phi(|k-m|)+\sum_{m\in P_{2}}\phi(|\bar{k}-m|)\\\ge\sum_{N_{1}+1}^{k=k_{0}}\sum_{m=m_{0}}^{N_{2}}\phi(|k-m|)+\sum_{m=m_{0}}^{N_{2}}\phi(|N_{1}-m|)=\sum_{N_{1}}^{k=k_{0}}\sum_{m=m_{0}}^{N_{2}}\phi(|k-m|). 
\end{align*}
as claimed.
 \end{proof}
 
We know $\widetilde{P}''_{1}(\varepsilon),\,\widetilde{P}''_{2}(\varepsilon)$ are disjoint and both are contained in the set of indices $\{-n''(\varepsilon),...,n''(\varepsilon)-1\}$. Then, if we apply Lemma \ref{rearrangement lemma} with $\phi(t)=\frac{1}{(t+1)^{2}}$
\[
\sum_{k\in\widetilde{P}''_{1}(\varepsilon)}\sum_{m\in\widetilde{P}''_{2}(\varepsilon)}\frac{1}{(|m-k|+1)^2}\ge\sum_{k=-n''(\varepsilon)}^{k_{0}}\sum_{m=m_{0}}^{n''(\varepsilon)-1}\frac{1}{(|m-k|+1)^2}
\]
where $k_{0}=\#\widetilde{P}''_{1}(\varepsilon)-n''(\varepsilon)-1$ and $m_{0}=n''(\varepsilon)-\#\widetilde{P}''_{2}(\varepsilon)$. We also have that $\#(P''_{3}(\e)\cup\bar{P}''_{3}(\varepsilon))$ is uniformly bounded independently $\e$ and $\tau$. This implies $\#P''_{0}(\varepsilon)\le N''$ for some $N''\ge3$. In particular, we have $m_{0}-k_{0}\le2N''-1$. Consequently, 
\begin{multline*}
    \sum_{k\in\widetilde{P}''_{1}(\varepsilon)}\sum_{m\in\widetilde{P}''_{2}(\varepsilon)}\frac{1}{(|m-k|+1)^2}\ge\sum_{k=-n''(\varepsilon)}^{k_{0}}\sum_{m=k_{0}+2N''}^{n''(\varepsilon)-1}\frac{1}{(|m-k|+1)^2}\\=\sum_{k=-N_{1}(\varepsilon)}^{-N''}\sum_{N''}^{N_{2}(\varepsilon)}\frac{1}{(|m-k|+1)^2}=\sum_{k=N''}^{N_{1}(\varepsilon)}\sum_{m=N''}^{N_{2}(\varepsilon)}\frac{1}{(m+k+1)^2},
\end{multline*}
where $N_{1}(\varepsilon):=\#\widetilde{P}''_{1}(\varepsilon)+N''-1$ and $N_{2}(\varepsilon):=2n''(\varepsilon)-\#\widetilde{P}''_{1}(\varepsilon)-N''$.
\begin{proposition}\label{harmonic sum}
The following hold:\\
{\rm(i)} $$\frac{1-4\tau}{1-2\tau}\le\liminf_{\varepsilon\to0}\frac{\#\widetilde{P}''_{1}(\varepsilon)}{n''(\varepsilon)}\le\limsup_{\varepsilon\to0}\frac{\#\widetilde{P}''_{1}(\varepsilon)}{n''(\varepsilon)}\le\frac{1}{1-2\tau},$$
\\
{\rm(ii)}
$$\liminf_{\varepsilon\to0}\frac{1}{\left|\log{\varepsilon}\right|}\sum_{k=N''}^{N_{1}(\varepsilon)}\sum_{m=N''}^{N_{2}(\varepsilon)}\frac{1}{(m+k+1)^2}\ge\theta''.$$
\end{proposition}
\begin{proof}
(i) We know that $\#(P''_{1}(\varepsilon)\setminus \widetilde{P}''_{1}(\varepsilon))\leq N''$, so in {\rm(i)} we can replace $\#\widetilde{P}''_{1}(\varepsilon)$ with $\#P''_{1}(\varepsilon).$ 
First, we claim that 
\[
\limsup_{\varepsilon\to0}\frac{\#P''_{1}(\varepsilon)}{n''(\varepsilon)}\le\frac{1}{1-2\tau},\qquad\limsup_{\varepsilon\to0}\frac{\#P''_{2}(\varepsilon)}{n''(\varepsilon)}\le\frac{1}{1-2\tau}.
\]
We prove the first inequality, as the second follows by the same argument. 
Observe that 
\[
\Big|\bigcup_{k\in P''_{1}(\varepsilon)}A''_{k}(\varepsilon)\Big|\le|\{t\in(t_{1}-\rho,t_{1}+\rho);\,u_{\varepsilon}(t)\le\tau-1\}|\to\rho\,\,\text{as}\,\varepsilon\to0,
\]
and
$\displaystyle\Big|\bigcup_{k\in P''_{1}(\varepsilon)}A''_{k}(\varepsilon)\Big|\ge(1-2\tau)\#P''_{1}(\varepsilon)\theta''(\e)$.
Hence, we obtain
\[
\limsup_{\varepsilon\to0}\frac{\#P''_{1}(\varepsilon)\theta''(\e)}{\rho}=\limsup_{\varepsilon\to0}\frac{\#P''_{1}(\varepsilon)}{n''(\varepsilon)}\le\frac{1}{1-2\tau}.
\]

Similarly, one can prove that
\[
\limsup_{\varepsilon\to0}\frac{\#P''_{2}(\varepsilon)\theta''(\e)}{\rho}=\limsup_{\varepsilon\to0}\frac{\#P''_{2}(\varepsilon)}{n''(\varepsilon)}\le\frac{1}{1-2\tau}.
\]
To complete the proof, we observe that $\lim_{\varepsilon\to0}\#(P''_{1}(\varepsilon)\cup P''_{2}(\varepsilon))/n''(\varepsilon)=2
$. It follows that
\[
\liminf_{\varepsilon\to0}\frac{\#P''_{1}(\varepsilon)}{n''(\varepsilon)}\ge2-\limsup_{\varepsilon\to0}\frac{\#P''_{2}(\varepsilon)}{n''(\varepsilon)}\ge\frac{1-4\tau}{1-2\tau}.
\]
which completes the proof of {\rm(i)}.

\smallskip
{\rm(ii)}. We now prove part {\rm(ii)} using part {\rm(i)}.
\begin{multline*}
    \sum_{k=N''}^{N_{1}(\varepsilon)}\sum_{m=N''}^{N_{2}(\varepsilon)}\frac{1}{(m+k+1)^2}\ge\sum_{k=N''}^{N_{1}(\varepsilon)}\sum_{m=N''}^{N_{2}(\varepsilon)}\frac{1}{(m+k+1)(m+k+2)}\\\ge\sum_{k=N''}^{N_{1}(\varepsilon)}(\frac{1}{k+N''+1}-\frac{1}{k+N_{2}(\varepsilon)+2})\ge-\Big(\frac{N_{1}(\varepsilon)}{N_{2}(\varepsilon)}+2N''\Big)+\sum_{k=1}^{N_{1}(\varepsilon)+N''+1}\frac{1}{k}.
\end{multline*}
As a result, we get
\begin{align*}
    \liminf_{\varepsilon\to0}\frac{1}{\left|\log{\varepsilon}\right|}\sum_{k=N''}^{N_{1}(\varepsilon)}\sum_{m=N''}^{N_{2}(\varepsilon)}\frac{1}{(m+k+1)^2}\ge\liminf_{\varepsilon\to0}\frac{1}{\left|\log{\varepsilon}\right|}\sum_{k=1}^{N_{1}(\varepsilon)+N''+1}\frac{1}{k}\\-\limsup_{\e\to0}\frac{1}{\left|\log{\varepsilon}\right|}\Big(\frac{N_{1}(\varepsilon)}{N_{2}(\varepsilon)}+2N''\Big).
\end{align*}
One can check that the second term tends to $0$. Then by {\rm(i)} 
\begin{multline*}
    \liminf_{\varepsilon\to0}\frac{1}{\left|\log{\varepsilon}\right|}\sum_{k=N''}^{N_{1}(\varepsilon)}\sum_{m=N''}^{N_{2}(\varepsilon)}\frac{1}{(m+k+1)^2}\ge\liminf_{\varepsilon\to0}\frac{1}{\left|\log{\varepsilon}\right|}\sum_{k=1}^{N_{1}(\varepsilon)+N''+1}\frac{1}{k}\\=\liminf_{\varepsilon\to0}\frac{\log{(\#\widetilde{P}''_{1}(\e)}/n''(\e))+\log{n''(\e)}}{\left|\log{\varepsilon}\right|}=\lim_{\varepsilon\to0}\frac{\log{n''(\varepsilon)}}{\left|\log{\varepsilon}\right|}=\theta'',
\end{multline*}
which completes the proof of \eqref{harmonic sum l.b}, from which we have \eqref{second-lb}. 

This completes the proof of the liminf inequality.
\end{proof}
\subsection{Proof of the limsup inequality}
This section is devoted to the proof of Theorem \ref{sub-main}(ii), namely the construction of a recovery sequence for the limsup inequality. The recovery sequence is constructed locally near each jump point of the limiting function, while it coincides with the original function away from the interfaces. We show that each transition layer contributes a fixed amount to the limit energy, and the result follows by summing these contributions over all jump points.

\smallskip
 Let $u\in\BV\left((0,1);\{-1,1\}\right)$ and let $\{t_{n}: n\in\{1,\ldots,\#S(u)
\}\}$ be the collection of jump points of $u$. Let $\rho>0$ be such that $\#\{(t_n-2\rho,t_n+2\rho)\cap\{t_m: 1\le m\le \#S(u)\}=1$ and $(t_n-2\rho,t_n+2\rho)
 \subset(0,1)$ for every $n\in\{1,...,\#S(u)\}$.  We fix a jump point $t_{n}$, without loss of generality we may assume $n=1$ and that the restriction of $u$ to $(t_{1}-\rho,t_{1}+\rho)$ is given by
\begin{equation*}
    u|_{(t_{1}-\rho,t_{1}+\rho)}(t)=
    \begin{cases}
     -1 & t\in(t_{1}-\rho,t_{1})\\
     1 & t\in(t_{1},t_{1}+\rho).
    \end{cases}
\end{equation*}
Otherwise, we replace $u$ with $-u$.  We now proceed with the construction of the recovery sequences in the two cases separately. 

\medskip
\noindent \textbf{Case 1.} Suppose $\lambda_{0}\in[0,1)$ and thus $\lambda=\lambda_{0}$. We then consider the following sequence of functions 
\begin{equation*}
    u^{1}_{\epsilon}(t):=
\begin{cases}
    -1 &t\in(t_{1}-\rho,t_{1}(\varepsilon)-\epsilon)\\
    \frac{t-t_{1}(\varepsilon)}{\epsilon}  &t\in[t_{1}(\varepsilon)-\epsilon,t_{1}(\varepsilon)+\epsilon]\\
    1 &t\in(t_{1}(\varepsilon)+\epsilon,t_{1}+\rho),
\end{cases}
\end{equation*}
where $t_{1}(\varepsilon)$ is given by $t_{1}(\varepsilon)=\delta(\varepsilon)\bar{t}+\delta(\varepsilon)\lfloor\frac{t_{1}}{\delta(\varepsilon)}\rfloor$ with $\bar{t}\in\{s:s=\arg\min_{t\in[0,1]} a(t,t)\}$. We claim that 
\begin{equation*}
    \lim_{\epsilon\to 0}F_{\varepsilon}(u^{1}_{\varepsilon},(t_{1}-\rho,t_{1}+\rho))=8\left((1-\lambda)a_{\min}+\lambda\bar{a}\right).
\end{equation*}

Since $W(-1)=W(1)=0$, we obtain the following.
\begin{align*}
    F_{\varepsilon}(u^{1}_{\varepsilon},(t_{1}-\rho,t_{1}+\rho))=\frac{1}{\varepsilon|\log{\varepsilon}|}\int_{t_{1}(\varepsilon)-\varepsilon}^{t_{1}(\varepsilon)+\varepsilon}W(u_{\epsilon}(t))dt\\+\frac{1}{|\log\varepsilon|}\int_{t_{1}-\rho}^{t_{1}+\rho}\int_{t_{1}-\rho}^{t_{1}+\rho}a\Big(\frac{s}{\delta(\varepsilon)},\frac{t}{\delta(\varepsilon)}\Big)\frac{\left|u^{1}_{\epsilon}(s)-u^{1}_{\epsilon}(t)\right|^2}{\left|s-t\right|^2}ds\,dt.
\end{align*}
Then we get 
\begin{equation*}
    F_{\varepsilon}(u^{1}_{\varepsilon},(t_{1}-\rho,t_{1}+\rho))=E({\varepsilon})+o(1)\,\,\text{as}\,\,\varepsilon\to0,
\end{equation*}
where 
\begin{equation}
    E({\varepsilon}):=\frac{1}{|\log\varepsilon|}\int_{t_{1}-\rho}^{t_{1}+\rho}\int_{t_{1}-\rho}^{t_{1}+\rho}a\Big(\frac{s}{\delta(\varepsilon)},\frac{t}{\delta(\varepsilon)}\Big)\frac{\left|u^{1}_{\varepsilon}(s)-u^{1}_{\varepsilon}(t)\right|^2}{\left|s-t\right|^2}ds\,dt.
\end{equation}
Note that \(u_\varepsilon^1\) is constant on the intervals $(t_1-\rho,t_1(\varepsilon)-\e)$ and
$(t_1(\varepsilon)+\varepsilon,t_1+\rho)$.
Hence, if \(s\) and \(t\) belong to the same one of these intervals, then $|u_\varepsilon^1(s)-u_\varepsilon^1(t)|=0$. It follows that such pairs \((s,t)\) do not contribute to the energy \(E(\varepsilon)\). Therefore, we may write
$   E({\varepsilon})=E_{1}(\varepsilon)+E_{2}(\varepsilon)+E_{3}(\varepsilon)+E_{4}(\varepsilon)$,
with 
\begin{equation*}
    E_{1}(\varepsilon)=\frac{2}{|\log\varepsilon|}\int_{t_{1}-\rho}^{t_{1}(\varepsilon)-\varepsilon}\int_{t_{1}(\varepsilon)+\varepsilon}^{t_{1}+\rho}a\Big(\frac{s}{\delta(\varepsilon)},\frac{t}{\delta(\varepsilon)}\Big)\frac{\left|u^{1}_{\varepsilon}(s)-u^{1}_{\varepsilon}(t)\right|^2}{\left|s-t\right|^2}ds\,dt,
\end{equation*}
\begin{equation*}
    E_{2}(\varepsilon)=\frac{1}{|\log\varepsilon|}\int_{t_{1}(\varepsilon)-\varepsilon}^{t_{1}(\varepsilon)+\varepsilon}\int_{t_{1}(\varepsilon)-\varepsilon}^{t_{1}(\varepsilon)+\varepsilon}a\left(\frac{s}{\delta(\varepsilon)},\frac{t}{\delta(\varepsilon)}\right)\frac{\left|u^{1}_{\varepsilon}(s)-u^{1}_{\varepsilon}(t)\right|^2}{\left|s-t\right|^2}ds\,dt,
\end{equation*}
\begin{equation*}
     E_{3}(\varepsilon)=\frac{2}{|\log\varepsilon|}\int_{t_{1}-\rho}^{t_{1}(\varepsilon)-\varepsilon}\int_{t_{1}(\varepsilon)-\varepsilon}^{t_{1}(\varepsilon)+\varepsilon}a\left(\frac{s}{\delta(\varepsilon)},\frac{t}{\delta(\varepsilon)}\right)\frac{\left|u^{1}_{\varepsilon}(s)-u^{1}_{\varepsilon}(t)\right|^2}{\left|s-t\right|^2}ds\,dt,
\end{equation*}
\begin{equation*}
     E_{4}(\varepsilon)=\frac{2}{|\log\varepsilon|}\int_{t_{1}(\varepsilon)-\varepsilon}^{t_{1}(\varepsilon)+\varepsilon}\int_{t_{1}(\varepsilon)+\varepsilon}^{t_{1}+\rho}a\left(\frac{s}{\delta(\varepsilon)},\frac{t}{\delta(\varepsilon)}\right)\frac{\left|u^{1}_{\varepsilon}(s)-u^{1}_{\varepsilon}(t)\right|^2}{\left|s-t\right|^2}ds\,dt.
\end{equation*}
We claim that each of the error terms $E_i(\varepsilon)$, for $i=2,3,4$, converges to zero as $\varepsilon \to 0$, that is, $\lim_{\varepsilon \to 0} E_i(\varepsilon) = 0$. 

We estimate $ E_{2}(\varepsilon)$. This term corresponds to the interaction within the $\varepsilon$-neighborhood of $t_1(\varepsilon)$. Using the boundedness of $a$, we obtain
\begin{equation*}
    E_{2}(\varepsilon)=\frac{1}{|\log\varepsilon|}\int_{t_{1}(\varepsilon)-\varepsilon}^{t_{1}(\varepsilon)+\varepsilon}\int_{t_{1}(\varepsilon)-\varepsilon}^{t_{1}(\varepsilon)+\varepsilon}a\Big(\frac{s}{\delta(\varepsilon)},\frac{t}{\delta(\varepsilon)}\Big)\frac{\left|u^{1}_{\varepsilon}(s)-u^{1}_{\varepsilon}(t)\right|^2}{\left|s-t\right|^2}ds\,dt\leq\frac{4\beta_{a}}{|\log\varepsilon|}
\end{equation*}

We estimate $ E_{3}(\varepsilon)$. This term describes the interaction between the left region $(t_1-\rho,\, t_1(\varepsilon)-\varepsilon)$ and the $\varepsilon$-neighborhood of $t_1(\varepsilon)$. Using the boundedness of $a$ and proceeding as in the previous estimate, we obtain the following.
\begin{multline*}
    E_{3}(\varepsilon)\leq
    \frac{2\beta_{a}}{|\log\varepsilon|}\int_{t_{1}-\rho}^{t_{1}(\varepsilon)-\varepsilon}\int_{t_{1}(\varepsilon)-\varepsilon}^{t_{1}(\varepsilon)+\varepsilon}\frac{\left|u^{1}_{\varepsilon}(s)-u^{1}_{\varepsilon}(t)\right|^2}{\left|s-t\right|^2}ds\,dt\\=\frac{2\beta_{a}}{\varepsilon^2|\log\varepsilon|}\int_{t_{1}-\rho}^{t_{1}(\varepsilon)-\varepsilon}\int_{t_{1}(\varepsilon)-\varepsilon}^{t_{1}(\varepsilon)+\varepsilon}\frac{\left(s-(t_{1}(\varepsilon)-\varepsilon)\right)^2}{\left|s-t\right|^2}ds\,dt\\=\frac{2\beta_{a}}{\varepsilon^2|\log\varepsilon|}\int_{t_{1}(\varepsilon)-\varepsilon}^{t_{1}(\varepsilon)+\varepsilon}\left(s-(t_{1}(\varepsilon)-\varepsilon)\right)^2\frac{1}{s-t}\Big|_{t_{1}-\rho}^{t_{1}(\varepsilon)-\varepsilon}ds\\\leq\frac{2\beta_{a}}{\varepsilon^2|\log\varepsilon|}\int_{t_{1}(\varepsilon)-\varepsilon}^{t_{1}(\varepsilon)+\varepsilon}\left(s-(t_{1}(\varepsilon)-\varepsilon)\right)ds=\frac{4\beta_{a}}{|\log\varepsilon|}.
\end{multline*}

We estimate $E_{4}(\varepsilon)$. This term represents the interaction between the $\varepsilon$-neigh\-bor\-hood of $t_1(\varepsilon)$ and the right region $(t_1(\varepsilon)+\varepsilon,\, t_1+\rho)$. The argument is analogous to that for $E_3(\varepsilon)$, and we obtain
\begin{equation*}
    E_{4}(\varepsilon)\leq\frac{2\beta_{a}}{|\log\varepsilon|}\int_{t_{1}(\varepsilon)-\varepsilon}^{t_{1}(\varepsilon)+\varepsilon}\int_{t_{1}(\varepsilon)+\varepsilon}^{t_{1}+\rho}\frac{\left|u^{1}_{\varepsilon}(s)-u^{1}_{\varepsilon}(t)\right|^2}{\left|s-t\right|^2}ds\,dt\le\frac{4\beta_{a}}{|\log\varepsilon|}.
\end{equation*}
These estimates are sufficient to conclude the claim.

As for $E_1(\varepsilon)$, we write
$
  E_{1}(\varepsilon)=E_{1,1}(\varepsilon)+E_{1,2}(\varepsilon)+E_{1,3}(\varepsilon)+E_{1,4}(\varepsilon),  
$
 with
 \begin{equation*}
  E_{1,1}(\varepsilon)= \frac{2}{|\log\varepsilon|}\int_{t_{1}(\varepsilon)-\delta(\varepsilon)}^{t_{1}(\varepsilon)-\varepsilon}\int_{t_{1}(\varepsilon)+\varepsilon}^{t_{1}(\varepsilon)+\delta(\varepsilon)}a\Big(\frac{s}{\delta(\varepsilon)},\frac{t}{\delta(\varepsilon)}\Big)\frac{\left|u^{1}_{\varepsilon}(s)-u^{1}_{\varepsilon}(t)\right|^2}{\left|s-t\right|^2}ds\,dt  
 \end{equation*}\begin{equation*}
  E_{1,2}(\varepsilon)= \frac{2}{|\log\varepsilon|}\int_{t_{1}-\rho}^{t_{1}(\varepsilon)-\delta(\varepsilon)}\int_{t_{1}(\varepsilon)+\delta(\varepsilon)}^{t_{1}+\rho}a\Big(\frac{s}{\delta(\varepsilon)},\frac{t}{\delta(\varepsilon)}\Big)\frac{\left|u^{1}_{\varepsilon}(s)-u^{1}_{\varepsilon}(t)\right|^2}{\left|s-t\right|^2}ds\,dt   
 \end{equation*}
 \begin{equation*}
  E_{1,3}(\varepsilon)=\frac{2}{|\log\varepsilon|}\int_{t_{1}-\rho}^{t_{1}(\varepsilon)-\delta(\varepsilon)}\int_{t_{1}(\varepsilon)+\varepsilon}^{t_{1}(\varepsilon)+\delta(\varepsilon)}a\Big(\frac{s}{\delta(\varepsilon)},\frac{t}{\delta(\varepsilon)}\Big)\frac{\left|u^{1}_{\varepsilon}(s)-u^{1}_{\varepsilon}(t)\right|^2}{\left|s-t\right|^2}ds\,dt   
 \end{equation*}
 \begin{equation*}
  E_{1,4}(\varepsilon)=\frac{2}{|\log\varepsilon|}\int_{t_{1}(\varepsilon)-\delta(\varepsilon)}^{t_{1}(\varepsilon)-\varepsilon}\int_{t_{1}(\varepsilon)+\delta(\varepsilon)}^{t_{1}+\rho}a\Big(\frac{s}{\delta(\varepsilon)},\frac{t}{\delta(\varepsilon)}\Big)\frac{\left|u^{1}_{\varepsilon}(s)-u^{1}_{\varepsilon}(t)\right|^2}{\left|s-t\right|^2}ds\,dt.   
 \end{equation*}
 By arguments analogous to those used for $E_3(\varepsilon)$ and $E_4(\varepsilon)$, we obtain
 \[
 \lim_{\varepsilon\to0}E_{1,3}(\varepsilon)=\lim_{\varepsilon\to0}E_{1,4}(\varepsilon)=0.
 \] 
 We now show that 
 $\lim\limits_{\varepsilon\to0}E_{1,1}(\varepsilon)=8(1-\lambda) a_{\min}$ and $\lim\limits_{\varepsilon\to0}E_{1,2}(\varepsilon)=8\lambda\bar{a}$.
 
We compute the limit of the term $E_{1,1}(\varepsilon)$ as follows:
 \begin{multline*}
     E_{1,1}(\varepsilon)= \frac{2}{|\log\varepsilon|}\int_{t_{1}(\varepsilon)-\delta(\varepsilon)}^{t_{1}(\varepsilon)-\varepsilon}\int_{t_{1}(\varepsilon)+\varepsilon}^{t_{1}(\varepsilon)+\delta(\varepsilon)}a\Big(\frac{s}{\delta(\varepsilon)},\frac{t}{\delta(\varepsilon)}\Big)\frac{\left|u^{1}_{\varepsilon}(s)-u^{1}_{\varepsilon}(t)\right|^2}{\left|s-t\right|^2}ds\,dt\\
     =\frac{8}{|\log\varepsilon|}\int_{t_{1}(\varepsilon)-\delta(\varepsilon)}^{t_{1}(\varepsilon)-\varepsilon}\int_{t_{1}(\varepsilon)+\varepsilon}^{t_{1}(\varepsilon)+\delta(\varepsilon)}\frac{a\big(\frac{s}{\delta(\varepsilon)},\frac{t}{\delta(\varepsilon)}\big)}{(s-t)^2}ds\,dt
     \\=\frac{8}{|\log\varepsilon|}\int_{\frac{\varepsilon}{\delta(\varepsilon)}}^{1}\int_{\frac{\varepsilon}{\delta(\varepsilon)}}^{1}\frac{a(\bar{t}+s,\bar{t}-t)}{(s+t)^2}ds\,dt.
 \end{multline*} 
 Hence,
\begin{equation*}
    \lim_{\varepsilon\to0}E_{1,1}(\varepsilon)=\lim_{\varepsilon\to0}\frac{8}{|\log\varepsilon|}\int_{\frac{\varepsilon}{\delta(\varepsilon)}}^{1}\int_{\frac{\varepsilon}{\delta(\varepsilon)}}^{1}\frac{a\left(\bar{t}+s,\bar{t}-t\right)}{(s+t)^2}ds\,dt.
\end{equation*}

To compute the limit of  $E_{1,1}(\varepsilon)$, we will use the following lemma.
\begin{lemma}
    Let $\alpha$ be a continuous function on $\mathbb{R}^{2}$ and $\delta(\varepsilon)\to0$ as $\varepsilon\to0$ such that
    $\displaystyle\lim_{\varepsilon\to0}\frac{\log{\delta(\varepsilon)}}{\log{\varepsilon}}=\lambda_{0}\in[0,1)$.
    Then, we have 
\begin{equation*}
    \lim_{\varepsilon\to 0}\frac{1}{|\log\varepsilon|}\int_{\frac{\varepsilon}{\delta(\varepsilon)}}^{1}\int_{\frac{\varepsilon}{\delta(\varepsilon)}}^{1}\frac{\alpha(s,t)}{(s+t)^2}ds\,dt=(1-\lambda_{0}) \alpha(0,0).
\end{equation*}
\end{lemma}

\begin{proof}
Since $\lambda_{0}\in[0,1)$ we have $\frac{\varepsilon}{\delta(\varepsilon)}\to0$ as $\varepsilon\to0$.
Let $c\in(\frac{\varepsilon}{\delta(\varepsilon)},1)$ be any fixed number. Then
  \begin{multline*}
\frac{\min_{s,t\in[\frac{\varepsilon}{\delta(\varepsilon)},c]}\alpha(s,t)}{|\log\varepsilon|}\int_{\frac{\varepsilon}{\delta(\varepsilon)}}^{c}\int_{\frac{\varepsilon}{\delta(\varepsilon)}}^{c}\frac{1}{(s+t)^2}dsdt\leq\frac{1}{|\log\varepsilon|}\int_{\frac{\varepsilon}{\delta(\varepsilon)}}^{1}\int_{\frac{\varepsilon}{\delta(\varepsilon)}}^{1}\frac{
      \alpha\left(s,t\right)}{(s+t)^2}dsdt\\\leq\frac{\max_{s,t\in[\frac{\varepsilon}{\delta(\varepsilon)},c]}\alpha(s,t)}{|\log\varepsilon|}\int_{\frac{\varepsilon}{\delta(\varepsilon)}}^{c}\int_{\frac{\varepsilon}{\delta(\varepsilon)}}^{c}\frac{1}{(s+t)^2}dsdt\\+\frac{2\max_{s,t\in[0,1]}\alpha(s,t)}{|\log\varepsilon|}\int_{c}^{1}\int_{\frac{\varepsilon}{\delta(\varepsilon)}}^{1}\frac{1}{(s+t)^2}dsdt.
  \end{multline*}
  
  If we pass to the limit on both sides as $\varepsilon\to 0$, we obtain
  \begin{equation*}
      (1-\lambda_{0})\min_{s,t\in[0,c]}\alpha(s,t)\leq\lim_{\varepsilon\to 0}\frac{1}{|\log\varepsilon|}\int_{\frac{\varepsilon}{\delta(\varepsilon)}}^{1}\int_{\frac{\varepsilon}{\delta(\varepsilon)}}^{1}\frac{\alpha\left(s,t\right)}{(s+t)^2}ds\,dt\leq(1-\lambda_{0})\max_{s,t\in[0,c]}\alpha(s,t),
  \end{equation*}
  since we have
  \[
  \lim_{\varepsilon\to0}\frac{1}{|\log\varepsilon|}\int_{\frac{\varepsilon}{\delta(\varepsilon)}}^{c}\int_{\frac{\varepsilon}{\delta(\varepsilon)}}^{c}\frac{1}{(s+t)^2}ds\,dt=1-\lambda_{0},\,\, \lim_{\e\to0} \frac{1}{|\log\varepsilon|}\int_{c}^{1}\int_{\frac{\varepsilon}{\delta(\varepsilon)}}^{1}\frac{1}{(s+t)^2}dsdt=0.
  \]
  Letting $c\to0$, we obtain the desired result.
\end{proof}

By the above lemma, applied with $\alpha(s,t)= a(\bar{t}+s,\bar{t}-t)$, we obtain
\begin{equation*}
    \lim_{\varepsilon\to0}E_{1,1}(\varepsilon)=8(1-\lambda_{0}) a(\bar{t},\bar{t})=8(1-\lambda_{0})a_{\min}
\end{equation*}

To compute the limit of $E_{1,2}(\varepsilon)$ we will use the following lemma.
\begin{lemma}\label{harmonic sum lemma}
    Let $\varepsilon>0$, and let $N(\varepsilon)\in\N$ be such that
    $\displaystyle\lim_{\varepsilon\to0}\frac{\log{N(\varepsilon)}}{|\log{\varepsilon}|}=\lambda$.
    Then we have
    \[
    \lim_{\varepsilon\to0}\frac{1}{|\log{\varepsilon}|}\int_{1}^{N(\varepsilon)}\int_{1}^{N(\varepsilon)}\frac{a(s+s_{0},t+t_{0})}{(s+t)^{2}}ds\,dt=\lambda\bar{a},
    \]
    where $s_{0}$ and $t_{0}$ may depend on $\varepsilon.$
\end{lemma}
\begin{proof}
    We estimate the integral from above and below as follows:
\begin{multline*}
  \sum_{k=2}^{N(\varepsilon)}\sum_{m=2}^{N(\varepsilon)}\frac{\bar{a}}{(m+k)(m+k+1)}\le\sum_{k=2}^{N(\varepsilon)}\sum_{m=2}^{N(\varepsilon)}\int_{k-1}^{k}\int_{m-1}^{m}\frac{a\left(s+s_{0},t+t_{0}\right)}{(s+t)^2}ds\,dt\\\le\sum_{k=2}^{N(\varepsilon)}\sum_{m=2}^{N(\varepsilon)}\frac{\bar{a}}{(m+k-2)(m+k-3)}.
\end{multline*}
Moreover, for $N(\varepsilon)\ge2$, we have 
\[
\sum_{k=2}^{N(\varepsilon)}\sum_{m=2}^{N(\varepsilon)}\frac{1}{(m+k)(m+k+1)}=\sum_{k=2}^{N(\varepsilon)}\left(\frac{1}{k+2}-\frac{1}{k+N(\varepsilon)+1}\right)\ge-4+\sum_{k=1}^{N(\varepsilon)}\frac{1}{k},
\]
and
\[
\sum_{k=2}^{N(\varepsilon)}\sum_{m=2}^{N(\varepsilon)}\frac{1}{(m+k-2)(m+k-3)}=\sum_{k=2}^{N(\varepsilon)}\left(\frac{1}{k}-\frac{1}{k+N(\varepsilon)-3}\right)\le\sum_{k=1}^{N(\varepsilon)}\frac{1}{k}.
\]
Therefore, we obtain
\[
-\frac{4\bar{a}}{|\log{\varepsilon}|}+\frac{\bar{a}}{|\log{\varepsilon}|}\sum_{k=1}^{N(\varepsilon)}\frac{1}{k}\le\frac{1}{|\log{\varepsilon}|}\int_{1}^{N(\varepsilon)}\int_{1}^{N(\varepsilon)}\frac{a(s+s_{0},t+t_{0})}{(s+t)^{2}}ds\,dt\le\frac{\bar{a}}{|\log{\varepsilon}|}\sum_{k=1}^{N(\varepsilon)}\frac{1}{k}.
\]
Using
$\sum\limits_{k=1}^{n}\frac{1}{k}\sim\log{n}$ as $n\to\infty$,
we obtain
\[
\lim_{\varepsilon\to0}\frac{1}{|\log{\varepsilon}|}\int_{1}^{N(\varepsilon)}\int_{1}^{N(\varepsilon)}\frac{a(s+s_{0},t+t_{0})}{(s+t)^{2}}ds\,dt=\bar{a} \lim_{\varepsilon\to0}\frac{\log{N(\varepsilon)}}{|\log{\varepsilon}|}=\bar{a}\lambda
\]
and the claim.
\end{proof}
We rewrite $E_{1,2}(\varepsilon)$ in the following form.
\begin{multline*}
   E_{1,2}(\varepsilon)= \frac{2}{|\log\varepsilon|}\int_{t_{1}-\rho}^{t_{1}(\varepsilon)-\delta(\varepsilon)}\int_{t_{1}(\varepsilon)+\delta(\varepsilon)}^{t_{1}+\rho}a\Big(\frac{s}{\delta(\varepsilon)},\frac{t}{\delta(\varepsilon)}\Big)\frac{\left|u^{1}_{\varepsilon}(s)-u^{1}_{\varepsilon}(t)\right|^2}{\left|s-t\right|^2}ds\,dt\\=\frac{8}{|\log\varepsilon|}\int_{t_{1}-\rho}^{t_{1}(\varepsilon)-\delta(\varepsilon)}\int_{t_{1}(\varepsilon)+\delta(\varepsilon)}^{t_{1}+\rho}\frac{a\big(\frac{s}{\delta(\varepsilon)},\frac{t}{\delta(\varepsilon)}\big)}{(s-t)^2}ds\,dt\\=\frac{8}{|\log\varepsilon|}\int_{1}^{\frac{\rho-(t_{1}-t_{1}(\varepsilon))}{\delta(\varepsilon)}}\int_{1}^{\frac{\rho+(t_{1}-t_{1}(\varepsilon))}{\delta(\varepsilon)}}\frac{a\big(\bar{t}+s,\bar{t}-t\big)}{(s+t)^2}ds\,dt.
\end{multline*}
By Lemma \ref{harmonic sum lemma}, it follows that 
\begin{align*}
\lim_{\varepsilon\to0}E_{1,2}(\varepsilon)=\lim_{\varepsilon\to0}\frac{8}{|\log\varepsilon|}\int_{1}^{\lfloor\frac{\rho-(t_{1}-t_{1}(\varepsilon))}{\delta(\varepsilon)}\rfloor}\int_{1}^{\lfloor\frac{\rho+(t_{1}-t_{1}(\varepsilon))}{\delta(\varepsilon)}\rfloor}\frac{a\left(\bar{t}+s,\bar{t}-t\right)}{(s+t)^2}ds\,dt\\=\lim_{\varepsilon\to0}\frac{8}{|\log\varepsilon|}\int_{1}^{\lfloor\frac{\rho+(t_{1}-t_{1}(\varepsilon))}{\delta(\varepsilon)}\rfloor}\int_{1}^{\lfloor\frac{\rho+(t_{1}-t_{1}(\varepsilon))}{\delta(\varepsilon)}\rfloor}\frac{a\left(\bar{t}+s,\bar{t}-t\right)}{(s+t)^2}ds\,dt\\=8\bar{a}\lim_{\varepsilon\to0}\frac{\log{\lfloor\frac{\rho+(t_{1}-t_{1}(\varepsilon))}{\delta(\varepsilon)}\rfloor}}{|\log\varepsilon|}=8\lambda\bar{a}.
\end{align*}
Thus, in this case, we have 
\[
\lim_{\epsilon\to 0}F_{\epsilon}(u_{\epsilon};(t_{1}-\rho,t_{1}+\rho))=8\left((1-\lambda)a_{\min}+\lambda\bar{a}\right).
\] 

\smallskip
\noindent \textbf{Case 2.} Suppose $\lambda_{0}\ge1$ and thus $\lambda=1$. Then, consider as a recovery sequence $(u^{1}_{\varepsilon})$ defined by 
\begin{equation*}
    u^{1}_{\epsilon}(t)=
\begin{cases}
    -1 &t\in(t_{1}-\rho,t_{1}-\varepsilon)\\
    \frac{t-t_{1}}{\varepsilon}  &t\in[t_{1}-\varepsilon,t_{1}+\varepsilon]\\
    1 &t\in(t_{1}+\varepsilon,t_{1}+\rho).
\end{cases}
\end{equation*}
By the same argument as above we can show that 
\[
F_{\epsilon}(u_{\epsilon};(t_{1}-\rho,t_{1}+\rho))
=E_{1}(\varepsilon)+o(1)\,\,\text{as}\,\,\varepsilon\to0,
\]
where 
\[
E_{1}(\varepsilon)=\frac{8}{|\log{\varepsilon}|}\int_{t_{1}-\rho}^{t_{1}-\varepsilon}\int_{t_{1}+\varepsilon}^{t_{1}+\rho}\frac{a\big(\frac{s}{\delta(\varepsilon)},\frac{t}{\delta(\varepsilon)})}{(s-t)^{2}}ds\,dt
\]

We show that 
$\lim\limits_{\varepsilon\to0}E_{1}(\varepsilon)=8\bar{a}
$.
 To this end, we proceed as follows:
\begin{multline*}
    \frac{1}{|\log{\varepsilon}|}\int_{t_{1}-\rho}^{t_{1}-\varepsilon}\int_{t_{1}+\varepsilon}^{t_{1}+\rho}\frac{a\big(\frac{s}{\delta(\varepsilon)},\frac{t}{\delta(\varepsilon)}\big)}{(s-t)^{2}}ds\,dt=\frac{1}{|\log{\varepsilon}|}\int_{\varepsilon}^{\rho}\int_{\varepsilon}^{\rho}\frac{a\big(\frac{t_{1}+s}{\delta(\varepsilon)},\frac{t_{1}-t}{\delta(\varepsilon)}\big)}{(s+t)^{2}}ds\,dt\\=\frac{1}{|\log{\varepsilon}|}\int_{\frac{\varepsilon}{\delta(\varepsilon)}}^{\frac{\rho}{\delta(\varepsilon)}}\int_{\frac{\varepsilon}{\delta(\varepsilon)}}^{\frac{\rho}{\delta(\varepsilon)}}\frac{a\big(\frac{t_{1}}{\delta(\varepsilon)}+s,\frac{t_{1}}{\delta(\varepsilon)}-t\big)}{(s+t)^{2}}ds\,dt\\=\frac{1}{|\log{\varepsilon}|}\int_{1}^{\frac{\rho}{\delta(\varepsilon)}}\int_{1}^{\frac{\rho}{\delta(\varepsilon)}}\frac{a\big(\frac{t_{1}}{\delta(\varepsilon)}+s,\frac{t_{1}}{\delta(\varepsilon)}-t\big)}{(s+t)^{2}}ds\,dt
     \\
    -\frac{1}{|\log{\varepsilon}|}\int_{1}^{\frac{\varepsilon}{\delta(\varepsilon)}}\int_{1}^{\frac{\varepsilon}{\delta(\varepsilon)}}\frac{a\big(\frac{t_{1}}{\delta(\varepsilon)}+s,\frac{t_{1}}{\delta(\varepsilon)}-t\big)}{(s+t)^{2}}ds\,dt
    \\
    -\frac{1}{|\log{\varepsilon}|}\int_{1}^{\frac{\varepsilon}{\delta(\varepsilon)}}\int_{\frac{\varepsilon}{\delta(\varepsilon)}}^{\frac{\rho}{\delta(\varepsilon)}}\frac{a\big(\frac{t_{1}}{\delta(\varepsilon)}+s,\frac{t_{1}}{\delta(\varepsilon)}-t\big)}{(s+t)^{2}}
    ds\,dt
    \\
    -\frac{1}{|\log{\varepsilon}|}\int_{\frac{\varepsilon}{\delta(\varepsilon)}}^{\frac{\rho}{\delta(\varepsilon)}}\int_{1}^{\frac{\varepsilon}{\delta(\varepsilon)}}\frac{a\big(\frac{t_{1}}{\delta(\varepsilon)}+s,\frac{t_{1}}{\delta(\varepsilon)}-t\big)}{(s+t)^{2}}ds\,dt.
\end{multline*}
One can show that each of the last two integrals converges to $0$ as $\e\to0$, since $a$ is bounded and 
\begin{align*}
    \lim_{\varepsilon\to0}\frac{1}{|\log{\varepsilon}|}\int_{1}^{\frac{\varepsilon}{\delta(\varepsilon)}}\int_{\frac{\varepsilon}{\delta(\varepsilon)}}^{\frac{\rho}{\delta(\varepsilon)}}\frac{1}{(s+t)^{2}}ds\,dt=0.
\end{align*}
 By Lemma \ref{harmonic sum lemma}, we obtain
 \begin{equation}\label{lambda}
     \lim_{\varepsilon\to0}\frac{1}{|\log{\varepsilon}|}\int_{1}^{\frac{\rho}{\delta(\varepsilon)}}\int_{1}^{\frac{\rho}{\delta(\varepsilon)}}\frac{a\big(\frac{t_{1}}{\delta(\varepsilon)}+s,\frac{t_{1}}{\delta(\varepsilon)}-t\big)}{(s+t)^{2}}ds\,dt=\lambda_{0}\bar{a}.
 \end{equation}
Moreover, if $\lambda_{0}>1$, then again by Lemma \ref{harmonic sum lemma}
\begin{equation}\label{lambda-1}
    \lim_{\varepsilon\to0}\frac{1}{|\log{\varepsilon}|}\int_{1}^{\frac{\varepsilon}{\delta(\varepsilon)}}\int_{1}^{\frac{\varepsilon}{\delta(\varepsilon)}}\frac{a\big(\frac{t_{1}}{\delta(\varepsilon)}+s,\frac{t_{1}}{\delta(\varepsilon)}-t\big)}{(s+t)^{2}}ds\,dt=(\lambda_{0}-1)\bar{a}.
\end{equation}
On the other hand, if $\lambda_{0}=1$, then a direct computation yields
\begin{multline*}
    \lim_{\e\to0}\frac{1}{|\log{\varepsilon}|}\int_{1}^{\frac{\varepsilon}{\delta(\varepsilon)}}\int_{1}^{\frac{\varepsilon}{\delta(\varepsilon)}}\frac{a\big(\frac{t_{1}}{\delta(\varepsilon)}+s,\frac{t_{1}}{\delta(\varepsilon)}-t\big)}{(s+t)^{2}}ds\,dt\\
    \le\lim_{\e\to0}\frac{\beta_{a}}{|\log{\varepsilon}|}\int_{1}^{\frac{\varepsilon}{\delta(\varepsilon)}}\int_{1}^{\frac{\varepsilon}{\delta(\varepsilon)}}\frac{1}{(s+t)^{2}}ds\,dt=0
\end{multline*}
Hence, equality \eqref{lambda-1} remains valid in the case $\lambda_{0}=1$. It follows from \eqref{lambda} and \eqref{lambda-1} that $
\lim\limits_{\varepsilon\to0}E_{1}(\varepsilon)=\lambda_{0}\bar{a}-(\lambda_{0}-1)\bar{a}=\bar{a}$.

\medskip
\noindent \textbf{Conclusion.} Let $u\in\BV((0,1);\{-1,1\})$ then we consider 
\begin{equation*}
u_{\varepsilon}(t)=
    \begin{cases}
        u(t) & \text{if}\,\,t\notin\bigcup_{n=1}^{\#S(u)}(t_{n}-\rho,t_{n}+\rho),\\
        u^{n}_{\varepsilon}(t) & \text{if}\,\,t\in(t_{n}-\rho,t_{n}+\rho).
    \end{cases}
\end{equation*}
Let $n\in\{1,\ldots,\#S(u)\}$ and let $l\in[t_n+\rho,\, t_{n+1}-\rho]$, where $t_{\#S(u)+1}=1$. In particular, if $n=\#S(u)$, then $l\in[t_{\#S(u)}+\rho,1]$. We then claim that
\begin{equation}\label{induction argument}
    \lim_{\varepsilon\to0}F_{\varepsilon}(u_{\varepsilon},(0,l))=8n((1-\lambda)a_{\min}+\lambda\bar{a}).
\end{equation}
We prove the claim by induction on $n$. Suppose $n=1$. Since $W(-1)=W(1)=0$, we obtain the following.
\begin{multline*}
    F_{\varepsilon}(u_{\varepsilon},(0,l))=F_{\varepsilon}(u_{\varepsilon},(t_{1}-\rho,t_{1}+\rho))\\+\frac{1}{\left|\log{\e}\right|}\int_{0}^{t_{1}-\rho}\int_{t_{1}-\rho}^{l}a\big(\frac{s}{\delta(\e) },\frac{t}{\delta(\e) }\big)\frac{\left|u_{\e}(s)-u_{\e}(t)\right|^2}{\left|s-t\right|^2}ds\,dt\\+\frac{1}{\left|\log{\e}\right|}\int_{t_{1}-\rho}^{t_{1}+\rho}\int_{0}^{t_{1}-\rho}a\big(\frac{s}{\delta(\e) },\frac{t}{\delta(\e) }\big)\frac{\left|u_{\e}(s)-u_{\e}(t)\right|^2}{\left|s-t\right|^2}ds\,dt\\+\frac{1}{\left|\log{\e}\right|}\int_{t_{1}-\rho}^{t_{1}+\rho}\int_{t_{1}+\rho}^{l}a\big(\frac{s}{\delta(\e) },\frac{t}{\delta(\e) }\big)\frac{\left|u_{\e}(s)-u_{\e}(t)\right|^2}{\left|s-t\right|^2}ds\,dt\\+\frac{1}{\left|\log{\e}\right|}\int_{t_{1}+\rho}^{l}\int_{0}^{t_{1}+\rho}a\big(\frac{s}{\delta(\e) },\frac{t}{\delta(\e) }\big)\frac{\left|u_{\e}(s)-u_{\e}(t)\right|^2}{\left|s-t\right|^2}ds\,dt.
\end{multline*}
Each of the last four integrals converges to zero as $\varepsilon\to0$. We show this for the first term; the other cases follow analogously.
We may assume that $\varepsilon$ is sufficiently small so that
$u_{\varepsilon}$ is constant on $(t_{1}-\rho,t_{1}-\rho/2)$ and takes the same value on $(0,t_{1}-\rho)$. Hence,
\begin{multline*}
    \limsup_{\e\to0}\frac{1}{\left|\log{\e}\right|}\int_{0}^{t_{1}-\rho}\int_{t_{1}-\rho}^{l}a\big(\frac{s}{\delta(\e) },\frac{t}{\delta(\e) }\big)\frac{\left|u_{\e}(s)-u_{\e}(t)\right|^2}{\left|s-t\right|^2}ds\,dt\\=\limsup_{\e\to0}\frac{1}{\left|\log{\e}\right|}\int_{0}^{t_{1}-\rho}\int_{t_{1}-\rho/2}^{l}a\big(\frac{s}{\delta(\e) },\frac{t}{\delta(\e) }\big)\frac{\left|u_{\e}(s)-u_{\e}(t)\right|^2}{\left|s-t\right|^2}ds\,dt\le\limsup_{\e\to0}\frac{16\beta_{a}}{\rho^{2}\left|\log{\e}\right|}.
\end{multline*}
Therefore, the first term converges to zero as $\varepsilon \to 0$. Combining Cases 1 and 2, we obtain
\[
 \lim_{\e\to0}F_{\varepsilon}(u_{\varepsilon},(0,l))=\lim_{\varepsilon\to0}F_{\varepsilon}(u^{1}_{\varepsilon},(t_{1}-\rho,t_{1}+\rho))=8((1-\lambda)a_{\min}+\lambda\bar{a}).
\]
Assume now that the claim holds for $k=n-1$. We show it for $k=n$. Indeed,
\begin{multline*}
     F_{\varepsilon}(u_{\varepsilon},(0,l))=F_{\varepsilon}(u_{\varepsilon},(0,t_{n-1}+\rho))+F_{\varepsilon}(u_{\varepsilon},(t_{n-1}+\rho,l))\\+\frac{1}{\left|\log{\e}\right|}\int_{0}^{t_{n-1}+\rho}\int_{t_{n-1}+\rho}^{l}a\big(\frac{s}{\delta(\e) },\frac{t}{\delta(\e) }\big)\frac{\left|u_{\e}(s)-u_{\e}(t)\right|^2}{\left|s-t\right|^2}ds\,dt\\+\frac{1}{\left|\log{\e}\right|}\int_{t_{n-1}+\rho}^{l}\int_{0}^{t_{n-1}+\rho}a\big(\frac{s}{\delta(\e) },\frac{t}{\delta(\e) }\big)\frac{\left|u_{\e}(s)-u_{\e}(t)\right|^2}{\left|s-t\right|^2}ds\,dt.
\end{multline*}
By the same argument as in the case $n=1$, one can show that the last two integrals vanish as $\varepsilon\to0$. Moreover, repeating the argument used in the case $n=1$, we get
\[
\lim_{\varepsilon\to0}F_{\varepsilon}(u_{\varepsilon},(t_{n-1}+\rho,l))=8((1-\lambda)a_{\min}+\lambda\bar{a}).
\]
Hence, we conclude that
\[
\lim_{\varepsilon\to0}F_{\varepsilon}(u_{\varepsilon},(0,l))=\lim_{\varepsilon\to0}F_{\varepsilon}(u_{\varepsilon},(0,t_{n-1}+\rho))+\lim_{\varepsilon\to0}F_{\varepsilon}(u_{\varepsilon},(t_{n-1}+\rho,l))=8n((1-\lambda)a_{\min}+\lambda\bar{a}).
\]
Then, if we choose $l=t_{\#S(u)+1}=1$
\[
\lim_{\varepsilon\to0}F_{\varepsilon}(u_{\varepsilon})=\lim_{\varepsilon\to0}F_{\varepsilon}(u_{\varepsilon},(0,1))=8\#S(u)((1-\lambda)a_{\min}+\lambda\bar{a})=F(u).
\]
This completes the proof of the limsup inequality.

\bigskip

\textbf{Acknowledgments.} This article is based on work supported by the Excellence Project MatMod@TOV (Grant Agreement No. CUP E83C23000330006) of the Department of Mathematics of the University of Rome Tor Vergata. AB is a member of GNAMPA, INdAM.

 \bibliographystyle{abbrv}
\bibliography{Bibliography}

\end{document}